\documentclass[11pt,reqno]{amsproc}

\title[Almost global existence for the equations]{Almost global existence for the Prandtl boundary layer equations}

\author[M.~Ignatova]{Mihaela Ignatova}
\author[V.~Vicol]{Vlad Vicol}

\address{Department of Mathematics, Princeton University, Princeton, NJ 08544}
\email{ignatova@math.princeton.edu}

\address{Department of Mathematics, Princeton University, Princeton, NJ 08544}
\email{vvicol@math.princeton.edu}

\usepackage[margin=1in]{geometry}
\usepackage{amsmath, amsthm, amssymb}
\usepackage{times}
\usepackage{graphicx}

\usepackage[usenames,dvipsnames,svgnames,table]{xcolor}
\usepackage[colorlinks=true, pdfstartview=FitV, linkcolor=blue, citecolor=blue, urlcolor=blue]{hyperref}

\theoremstyle{plain}
\newtheorem{theorem}{Theorem}[section]
\newtheorem{definition}[theorem]{Definition}
\newtheorem{lemma}[theorem]{Lemma}

\theoremstyle{definition}
\newtheorem{remark}[theorem]{Remark}

\numberwithin{equation}{section}

\renewcommand{\tilde}{\widetilde}

\def\RR{{\mathbb R}}
\def\HH{{\mathbb H}}
\def\HCal{{\mathcal H}}
\def\OO{{\mathcal O}}
\def\uu{{U}}
\def\eps{{\varepsilon}}
\def\phi{{\varphi}}
\def\tt{{\langle t \rangle}}
\def\ss{{\langle s \rangle}}
\def\yy{{y}}

\def\DWF{{\rm DawsonF}}
\def\erf{{\rm erf}}
\def\ww{\theta_\alpha}
\def\gnu{g^{(\nu)}}
\def\unu{u^{(\nu)}}
\def\vnu{v^{(\nu)}}

\def\uu{{\mathrm U}}
\def\vv{{\mathrm V}}

\begin{document}


\begin{abstract}
We consider the Prandtl boundary layer equations on the half plane, with initial datum that lies in a weighted $H^1$ space with respect to the normal variable, and is real-analytic with respect to the tangential variable. The boundary trace of the horizontal Euler flow is taken to be a constant. We prove that if the Prandtl datum lies within $\eps$ of a stable profile, then the unique solution of the Cauchy problem can be extended at least up to time $T_\eps \geq \exp(\eps^{-1}/ \log(\eps^{-1}))$ .
\end{abstract}



\maketitle


\section{Introduction}

We consider the two dimensional Prandtl boundary layer equations for the velocity field $(u^P,v^P)$
\begin{align}
&\partial_t u^P - \partial_y^2 u^P + u^P \partial_x u^P + v^P \partial_y u^P = - \partial_x p^E \label{eq:P:1} \\
&\partial_x u^P + \partial_y v^P = 0 \label{eq:P:2}
\end{align}
posed in the upper half plane $\HH = \{ (x,y) \in \RR^2 \colon y > 0\}$.  Here $p^E$ denotes the trace at $\partial \HH$ of the underlying Euler pressure. 
The boundary conditions
\begin{align} 
&u^P|_{y=0} = v^P|_{y=0} = 0 \label{eq:P:BC:1} \\
&u^P|_{y=\infty} = u^E \label{eq:P:BC:2} 
\end{align}
are obtained by matching the Navier-Stokes no-slip boundary condition $u^{NS}=0$ on $\partial \HH$, with the Euler slip boundary condition at $y=\infty$. The trace at $\partial \HH$ of the Euler tangential velocity $u^E$, obeys Bernoulli's law
\begin{align*} 
\partial_t u^E + u^E \partial_x u^E + \partial_x p^E = 0.
\end{align*}
The Prandtl system \eqref{eq:P:1}--\eqref{eq:P:BC:2} is  supplemented with a compatible initial condition
\begin{align} 
u^P|_{t=0} = u_0^P. \label{eq:P:IC}
\end{align}
Our main result states that if the Euler data $(u^E,p^E)$ is constant, and if the initial datum $u_0^P$ of the Prandtl equations lies within $\eps$ of the error function $\erf(y/2)$ (in a suitable topology), then the Prandtl equations have a unique (classical in $x$ weak in $y$) solution on $[0,T_\eps]$, where $T_\eps \geq \exp(\eps^{-1}/\log(\eps^{-1}))$.
\begin{theorem}[\bf Almost global existence]
\label{thm:main}
Let the Euler data be given by $u^E = \kappa $ and $\partial_x p^E=0$. Define 
\begin{align*}
u_0(x,y) = u_0^P(x,y) - \kappa\; \erf\left( \frac y 2 \right)
\end{align*}
where $\erf$ is the Gauss error function.
There exists a sufficiently large universal constant $C_*>0$ and a sufficiently small universal constant $\eps_*>0$ such that the following holds. For any given $\eps \in (0,\eps_*]$, assume that there exists an analyticity radius $\tau_0>0$ such that 
\begin{align*} 
\frac{C_*}{\log \frac 1 \eps} \leq \tau_0^{3/2}  \leq \frac{1}{C_* \eps^3},
\end{align*}
and such that the function
\begin{align*} 
g_0(x,y) =  \partial_y u_0(x,y) + \frac{y}{2} u_0(x,y)
\end{align*}
obeys
\begin{align*} 
\|g_0\|_{X_{2 \tau_0,1/2}} := \sum_{m\geq 0} \| \exp(\tfrac{y^2}{8}) \partial_x^m g_0(x,y)\|_{L^2(\HH)} (2 \tau_0)^{m} \frac{\sqrt{m+1}}{m!} \leq \eps.
\end{align*}
Then there exists a unique solution
$u^P$ of the Prandtl boundary layer equations on $[0,T_\eps]$, where
\begin{align*} 
T_\eps \geq \exp\left(\frac{\eps^{-1}}{\log (\eps^{-1})}\right).
\end{align*}
The solution $u^P$ is real analytic in $x$, with analyticity radius larger than $\tau_0/2$, and lies in a weighted $H^2$ space with respect to $y$. We emphasize that $\eps$ and $\tau_0$ are independent of $\kappa$.
\end{theorem}

The precise function spaces, in which the solution $u^P$ lies, are given in Theorem~\ref{thm:rigorous} below. The condition relating $\eps$ and $\tau_0$ stated above roughly speaking says that we think of $0 < \eps \ll 1$, and of $\tau_0 = \OO(1)$. The stated condition is the sharp version of this heuristic.

\begin{remark}[\bf Initial vorticity may change sign]
\label{rem:monotone}
We note that the initial datum $u_0^P$ is not necessarily monotonic in $y$, i.e. we do not necessarily have $\omega_0^P := \partial_y u_0^P \geq 0$ or $\leq 0$ on $\HH$.
Thus, the initial data in Theorem~\ref{thm:main} need not fit in the Oleinik~\cite{Oleinik66} sign-definite vorticity setting. To see this, one may for example consider $\kappa = \eps >0$ sufficiently small and $\tau_0 = 1/4$. We then let
\[
u_0^P(x,y) = \eps \left( \exp(-x^2)  \eta(y) + \erf(y/2) \right),
\] 
with $\eta(y)$ such that $\eta(0)=0$ and $\exp(y^2/4) \eta(y) \in L^\infty_y$. Then 
 \[
 \partial_y u_0^P(0,y) = \eps\left( \eta'(y) + \exp(-y^2/4)/\sqrt{\pi}\right) 
\]
 can be designed so that 
\[
\partial_y u_0^P (0,0) > 0 \qquad \mbox{and} \qquad \partial_y u_0^P (0,1) < 0.
\]
This indeed shows that the initial profiles considered in Theorem~\ref{thm:main} need not be monotonic in $y$. 
\end{remark}

\subsection{The local well-posedness of the Prandtl equations}

Before discussing the proof of our main result (cf.~Subsection~\ref{sec:proof} below), we present the history of the problem. The Prandtl equations arise from matched asymptotic expansions~\cite{Prandtl1904} meant to describe the boundary behavior of solutions to the Navier-Stokes equations with Dirichlet boundary conditions
\begin{alignat}{2}
&\partial_tu^{(\nu)} - \nu\Delta u^{(\nu)} + u^{(\nu)}\cdot\nabla u^{(\nu)} + \nabla p^{(\nu)}=0, \qquad \nabla \cdot u^{(\nu)}=0,
& &\qquad \mbox{in } \Omega, \label{eq:NSE} \\
&u^{(\nu)}=0,
& &\qquad \mbox{on }\partial\Omega
\notag
\end{alignat}
in the vanishing viscosity limit $\nu \to 0$. Here $\Omega \subset \RR^2$ is a smooth domain.
Formally, as $\nu \to 0$ the Navier-Stokes equations reduce to the Euler equations, for which the slip boundary condition $u^E \cdot n=0$ holds on $\partial\Omega$. Due to this mismatch of boundary conditions, uniform in $\nu$ bounds for $\nabla u^{(\nu)}$ in e.g. the $L^1(\Omega)$ norm, on an $\OO(1)$ time interval, remain an outstanding mathematical challenge.

One of the fundamental questions which arise is to either prove that the Prandtl asymptotic expansion 
\begin{align}
(u^{(\nu)},v^{(\nu)})(t,x,y) = (u^{E} ,v^{E})(t,x,y) + (u^{P} , \sqrt{\nu} v^{P})(t,x,y/\sqrt{\nu})  + o(\sqrt{\nu}),
\label{eq:expansion}
\end{align}
can be justified rigorously~\cite{SammartinoCaflisch98a,Maekawa14,GuoNguyen14}, or to show that it fails~\cite{Grenier00,GerardVaretDormy10,GuoNguyen11,GerardVaretNguyen12,GrenierGuoNguyen14,GrenierGuoNguyen14b,GrenierGuoNguyen14c}. Naturally, the answer  is expected to depend on the topology in which \eqref{eq:expansion} is considered, and this is intimately related to the question of well-posedness of the Prandtl system. By now, the local in time well-posedness of \eqref{eq:P:1}--\eqref{eq:P:IC} has been considered by many authors, see e.g.~\cite{Oleinik66, EEngquist97,SammartinoCaflisch98a,CaflischSammartino00,Grenier00b,CannoneLombardoSammartino01,HongHunter03,LombardoCannoneSammartino03, GarganoSammartinoSciacca09,MasmoudiWong12a,KukavicaVicol13a,GerardVaretMasmoudi13,Maekawa14,AlexandreWangXuYang14,KukavicaMasmoudiVicolWong14,LiuWangYang14,WangXieYang14,LiWuXu15} and references therein. However, the question of whether the inviscid limit $u^{(\nu)} \to u^E$ holds whenever the Prandtl equations are locally well-posed, and are thus stable in some sense, remains open. See~\cite{ConstantinKukavicaVicol14} for partial progress in this direction.

In~\cite{Oleinik66}, Oleinik proved the existence of solutions for the unsteady Prandtl system provided the prescribed horizontal velocities are positive and monotonic, i.e. $u^{E}>0$ and $\omega^P = \partial_y u^{P}>0$. From the physical point of view, the monotonicity assumption has stabilizing effect since it prevents boundary layer separation. The main ingredient of the proof is the Crocco transform which uses $u^P$ as an independent variable instead of $y$ and $\omega^P$ as an unknown instead of $u^P$. 
More recently in~\cite{MasmoudiWong12a}, Masmoudi and Wong use solely energy methods and a new change of variables to prove the local in time existence and uniqueness in weighted Sobolev spaces, under the Oleinik's monotonicity assumption. The main idea of~\cite{MasmoudiWong12a} is to use a Sobolev energy in terms of the good unknown $g^P = \omega^P - u^P \partial_y \log( \omega^P)$, which may be done if $\omega^P > 0$. The equation obeyed by the top derivative in $x$ of $g^P$ is better behaved than that of the top derivative of either $u^P$ or $\omega^P$.  Although cf.~Remark~\ref{rem:monotone} this change of variables is unavailable to us, the idea of a good unknown inspired by \cite{MasmoudiWong12a}, plays a fundamental role in our proof. Recently, in~\cite{KukavicaMasmoudiVicolWong14}, the local existence and uniqueness for the Prandtl system was proven for initial data with multiple monotonicity regions, as long as on the complement of these regions the initial datum is tangentially real-analytic.

For real-analytic initial datum, Sammartino and Caflisch \cite{SammartinoCaflisch98a,CaflischSammartino00} established the local well-posedness by using the abstract Cauchy-Kowalewski theorem. For initial datum analytic only with respect to the tangential variable the local well-posedness was obtained in~\cite{LombardoCannoneSammartino03}. 
In~\cite{KukavicaVicol13a}, the authors gave an energy-based proof of this fact, and considered initial data with polynomial rather than exponential matching at the top of the boundary layer $u^{P}(t,x,y)-u^{E}(t,x) \to 0$ as $y\to\infty$. The tangentially-analytic norms introduced in \cite{KukavicaVicol13a} encode at $L^2$ level a full one-derivative gain from the decaying analyticity radius. These norms play an essential role in our proof. 

More recently, for data in the Gevrey class-$7/4$ in the tangential variable,
which has a single curve of non-degenerate critical points (i.e. $\omega_0^P=0$ iff $y=a_0(x)>0$ with
$\partial_y \omega_0^P(x,a_0(x))>0$ for all $x$),
the local well-posedness of the Prandtl equations was proven by Gerard-Varet and Masmoudi in~\cite{GerardVaretMasmoudi13}. Note that this Gevrey-exponent is not in contradiction with the ill-posedness in Sobloev spaces established by Gerard-Varet and Dormy in~\cite{GerardVaretDormy10} 
for the linearized Prandtl equation around a non-monotonic shear flow. Here the authors show  that some perturbations with high tangential frequency, $k \gg 1$, grow in time as $e^{\sqrt{k}t}$. At the nonlinear level this strong ill-posedness was obtained by Gerard-Varet and Nguyen~\cite{GerardVaretNguyen12}.

\subsection{The long time behavior of the Prandtl equations}
As pointed out by Grenier, Guo, and Nguyen \cite{GrenierGuoNguyen14c,GrenierGuoNguyen14,GrenierGuoNguyen14b} (see also~\cite{DrazinReid04}), in order to make progress towards proving or disproving the inviscid limit of the Navier-Stokes equations, a finer understanding of the Prandtl equations is required, and in particular one must understand its behavior on a longer time interval than the one which causes the instability used to prove ill-posedness. However, to the best of our knowledge the long-time existence of the Prandtl equations has only been considered in~\cite{Oleinik66}, \cite{XinZhang04}, and~\cite{ZhangZhang14}.

Oleinik shows in \cite{Oleinik66} that global regular solutions exist,  when the horizontal variable $x$ belongs to a finite interval $[0,L]$, with $L$ sufficiently small. Xin and Zhang prove in~\cite{XinZhang04} that if the pressure gradient has a favorable sign, that is $\partial_x p^E(t,x)\le 0$ for all $t>0$ and $x\in \RR$,  and the initial condition $u_0^P$ of Prandtl is monotone in $y$, the solutions are global. However these are weak solutions in Crocco variables, which are not known to be unique or to be regular. The global existence of smooth solutions in the monotonic case remains to date open. In the case of large datum, the assumption of a monotone initial velocity is essential. Indeed, E and Engquist~\cite{EEngquist97} take $u^E = \partial_x p^E = 0$ and construct an initial datum $u_0^P$ which is real-analytic in the tangential variable, but for which $\omega_0^P$ is not sign definite, and prove that the resulting solution of Prandtl (known to exist for short time in view of~\cite{LombardoCannoneSammartino03,KukavicaVicol13a}) blows up in finite time. We emphasize that for this blowup to occur, the initial datum must be at least $\OO(1)$: indeed, for initial datum that is sufficiently small, the conditions of Lemma 2.1 in \cite{EEngquist97} fail, and thus the proof does not apply.

In fact, for initial datum that is tangentially real-analytic and small, in a recent paper Zhang and Zhang~\cite{ZhangZhang14} prove that the system \eqref{eq:P:1}--\eqref{eq:P:IC} has a unique solution on a time interval that is much longer than the one guaranteed by the local existence theory. More precisely, for $u^E = \eps$ and $\partial_x p^E = 0$ it is proven in~\cite{ZhangZhang14} that if $u_0^P = \OO(\eps)$ in a norm that encodes Gaussian decay as $y \to \infty$ and tangential analyticity in $x$, and if  $\eps \ll 1$, the time of existence of the resulting solution is at least $\OO(\eps^{-4/3})$. The elegant proof relies on anisotropic Littlewood-Paley energy estimates in tangentially analytic norms, inspired by the ones previously used by Chemin, Gallagher, and Paicu~\cite{CheminGallagherPaicu11} to treat the Navier-Stokes equations with datum highly oscillating in one direction (see also~\cite{PaicuZhang14} and references therein).

\subsection{Almost global existence for the Prandtl equations}
\label{sec:proof}
In~\cite[Remark 1.1]{ZhangZhang14}, the authors raise the following question: ``whether the lifespan obtained in Theorem 1.1 is sharp is a very interesting question''. That is, do the solutions of the Prandtl equations with size $\eps$ initial datum live for a time interval longer than $\OO(\eps^{-4/3})$? In this paper we give a positive answer to this question, and prove (cf.~Theorem~\ref{thm:main} or Theorem~\ref{thm:rigorous}) that in 2D we have almost global existence (in the sense of~\cite{Klainerman83}). That is, the solution lives up to time $\OO(\exp(\eps^{-1}/\log\eps^{-1}))$. Our initial datum $u_0^P$ consists of a stable $\OO(\kappa)$ boundary layer lift profile, and an $\OO(\eps)$ possibly unstable, but tangentially real-analytic profile. In particular, the total initial vorticity is not necessarily positive (cf.~Remark~\eqref{rem:monotone}). Whether solutions arising from sufficiently small initial datum are in fact global in time remains open, and this may depend on whether $\kappa \lesssim \eps$ or $\eps \ll \kappa$.

The proof of Theorem~\ref{thm:main} proceeds in several steps. In order to homogenize the boundary condition at $y=\infty$, we write  $u^P$ as a perturbation $u$ of a stable shear profile $\kappa \phi(t,y)$, with $\partial_y\phi(t,y)>0$. In order to capture the maximal time decay from the heat equation, and to explore certain cancellations in the nonlinear terms of Prandtl, we choose the boundary lift $\phi(t,y)$ to be the Gauss error function $\erf(y/\sqrt{4(t+1)})$. The equation obeyed by $u$ (cf.~\eqref{eq:New:P:velocity}) contains the usual terms in the Prandtl equations, but also terms that are linear and quadratic in $\phi$. The lift $\phi$ is chosen so that the quadratic terms in $\phi$ vanish, and we are left to understand the linear ones in $\phi$. We note that until here a similar path is followed in~\cite{ZhangZhang14}, and that energy estimates for the ensuing linear problem leads to the maximal time of existence $\OO(\eps^{-4/3})$. 

The main enemy to obtaining a longer time of existence is the term $\kappa v \partial_y \phi$ in the velocity equation \eqref{eq:New:P:velocity} (respectively $\kappa v \partial_y^2 \phi$ in the vorticity equation \eqref{eq:New:P:vorticity}). As $v=-\partial_y^{-1}\partial_x u$, this term loses one tangential derivative, and is merely linear in $u$, so that it is not small with respect to $\eps$. 

The main idea of our paper is to introduce a new {\em linearly-good unknown} $g = \omega - u \partial_y ( \log \partial_y \phi)$, cf.~\eqref{eq:g} below. This change of variable is directly motivated by the one in~\cite{MasmoudiWong12b} for the case $\omega^P >0$. The upshot is that $g$ obeys an equation in which the bad terms $\kappa v \partial_y \phi$ and $\kappa v \partial_y^2 \phi$ cancel out, cf.~\eqref{eq:g:1} below. The solution of this new equation may be shown to be globally well-posed. Note here that we may recover $u=\uu(g)$ and $v=\vv(g)$ via linear operators $\uu$ and $\vv$ that are nonlocal in $y$, cf.~\eqref{eq:def:uu}--\eqref{eq:def:vv}. Thus $g$ is the only prognostic variable in the problem, and the system \eqref{eq:g:1}--\eqref{eq:g:3} is equivalent to \eqref{eq:P:1}--\eqref{eq:P:BC:2}. 

In order to take advantage of the time decay in the heat equation, it is natural to replace the $y$ variable with the heat self-similar variable $z = z(t,y) = y/\sqrt{t+1}$, and to use $L^2$ norms with gaussian weights in the normal direction. The gaussian weights are useful when bounding $u$ and $v$ in terms of $g$, cf.~Lemma~\ref{lem:u:v:g}. Moreover, the gaussian weights allow us to deal with another technical obstacle; namely, that in unbounded domains the Poincar\'{e} inequality does not hold. However, with the gaussian weights defined in \eqref{eq:theta:def}, we may use a special case of the Treves inequality (cf.~Lemma~\eqref{lem:Carleman}) as a replacement of the Poincar\'e inequality. The need for a Poincar\'e-type inequality in the $y$ variable comes from the desire to work with $L^2$ norms, and still capture the full one-derivative gain from the decay of the analyticity radius. As was shown in~\cite{OliverTiti01,KukavicaVicol11a,KukavicaVicol13a} this may be achieved by designing norms based on $\ell^1$ rather than $\ell^2$ sums over the derivatives. The one derivative gain inherent in these $\ell^1$-based norms allows for direct energy estimates.

The main ingredient of the proof is the a priori estimate \eqref{eq:X:1} below. The idea is to solve the PDE \eqref{eq:g:1}--\eqref{eq:g:3} for $g$ simultaneously with a nonlinear ODE~\eqref{eq:tau:1} for the tangential analyticity radius $\tau$. The fact that the analyticity radius $\tau$ does not decrease to less than $\tau(0)/2$ on the time interval considered follows from the time integrability of the dissipative terms present on the left side of \eqref{eq:X:1} (see also~\cite{PaicuVicol11}). 

The proof of Theorem~\ref{thm:main} is concluded once we establish the uniqueness of solutions in this class, and show that there exists at least one solution to the coupled system for $g$ and $\tau$. While uniqueness follows from the available a priori estimates, the existence of solutions introduces a number of additional difficulties. One of these is proving existence of solutions to \eqref{eq:tau:1}. This is a first order ODE in $\tau$, for which the nonlinear forcing term is well defined (i.e., the infinite sum converges), only if the solution $g$ already is known to have analyticity radius $\tau$. To overcome this difficulty, we consider a dissipative approximation of \eqref{eq:g:1}--\eqref{eq:g:3} and for  $\nu>0$ add a $-\nu \partial_x^2 g$ term on the left side of \eqref{eq:g:1}. We prove that this regularized equation has solutions $g^{(\nu)}$ in a fixed order Sobolev space, and a posteriori show that these solutions are tangentially real-analytic, with radii $\tau^{(\nu)}$ that obey an ODE similar to \eqref{eq:tau:1}. Here we essentially use that the initial datum $g_0$ is assumed to have tangential analyticity radius $2\tau_0$, while the solution is only shown to have a radius that lies between $\tau_0/2$ and $\tau_0$. We prove that these radii $\tau^{(\nu)}$ are uniformly equicontinuous in $\nu$ (in fact uniformly H\"older $1/2$) so that they converge along a subsequence on the compact time interval $[0,T_\eps]$. To conclude the proof, we show that along this subsequence the $g^{(\nu)}$ are a Cauchy sequence when measured in the tangentially analytic norms, and that the limiting solution $g$ and limiting radius $\tau$ obey \eqref{eq:tau:1}.

\subsection*{Organization of the paper} The detailed reformulation of the Prandtl system is given in Section~\ref{sec:new}. Here we also define the spaces in which the solutions lives, and reformulate Theorem~\ref{thm:main} in these terms. The a priori estimates are given in Section~\ref{sec:apriori}, the uniqueness of solutions is proven in Section~\ref{sec:uniqueness}, and the details concerning the existence of solutions are given in Section~\ref{sec:existence}.

\section{The linearly-good unknown, function spaces, and the main result}
\label{sec:new}

Denote by
\begin{align} 
z = z(t,y) = \frac{\yy}{\tt^{1/2}}, \qquad \mbox{where} \qquad \tt = t+1
\label{eq:z:def}
\end{align}
the heat self-similar variable.
We consider the lift 
\[
\kappa \phi = \kappa \phi(t,y)
\]
of the boundary conditions \eqref{eq:P:BC:1}--\eqref{eq:P:BC:2}, where 
\begin{align} 
\phi(t,y) =   \Phi( z(t,y)) 
\label{eq:phi:def:0}
\end{align}
and the function $\Phi(z)$ obeys
\begin{align} 
\Phi(0) = 0, \qquad \lim_{z\to \infty} \Phi(z) = 1,  \qquad \Phi'(z) > 0.
\label{eq:phi:def:1}
\end{align}
We make a precise choice of $\Phi$ in \eqref{eq:phi:def} below. We already note that by design $\partial_y \phi(t,y)>0$ for $y > 0$, i.e., the vorticity of the shear flow $\phi$ is positive, and thus stable (in the sense of~\cite{Oleinik66}).

We write the solution of \eqref{eq:P:1}--\eqref{eq:P:BC:2} as a perturbation $u(t,x,y)$ of the lift $\kappa \phi(t,y)$ via
\begin{align*} 
u^P(t,x,y) = \kappa \phi(t,y) + u(t,x,y)
\end{align*}
so that the perturbation $u$ obeys the homogenous boundary conditions
\begin{align} 
u|_{y=0} = u|_{y= \infty} = 0 \label{eq:New:P:BC:1},
\end{align}
and satisfies the equation
\begin{align} 
&\partial_t u - \partial_y^2 u +  \kappa \phi  \partial_x u + \kappa v \partial_y \phi  + u \partial_x u+ v \partial_y u
= - \kappa (\partial_t \phi - \partial_y^2\phi),
\label{eq:New:P:velocity}
\end{align}
where $v$ is computed from $u$ as
\begin{align}
v(t,x,y) = - \int_0^y \partial_x u(t,x,\bar y) d\bar y.
\label{eq:v:u}
\end{align}
Using \eqref{eq:New:P:velocity}, we obtain the equation for the perturbed vorticity $\omega= \partial_y u$,
\begin{align} 
\partial_t \omega - \partial_y^2 \omega + \kappa \phi \partial_x \omega + \kappa v \partial_y^2 \phi + u \partial_x \omega + v \partial_y \omega = - \kappa( \partial_t \partial_y \phi - \partial_y^3 \phi),
\label{eq:New:P:vorticity}
\end{align}
with the natural boundary condition
\begin{align} 
\partial_y \omega |_{y=0} = - \kappa \partial_y^2 \phi|_{y=0}. \label{eq:New:P:BC:2}
\end{align}
As we work in weighted spaces, at $y=\infty$ we impose the condition  
\begin{align}
\omega|_{y=\infty} = 0.
\label{eq:New:P:BC:3}
\end{align} 

\subsection{The linearly-good unknown}
As can already be seen in~\cite{ZhangZhang14}, the main obstruction for obtaining the global in time existence of solutions comes from the linear problem for the velocity  and  vorticity  
\begin{align} 
&\partial_t u - \partial_y^2 u +  \kappa \phi  \partial_x u  + \kappa v  \partial_y \phi  
= - \kappa (\partial_t \phi - \partial_y^2\phi) + \mbox{nonlinearity}, \label{eq:toy:u}\\
&\partial_t \omega  - \partial_y^2 \omega  + \kappa \phi \partial_x \omega  + \kappa v  \partial_y^2 \phi  = - \kappa( \partial_t \partial_y \phi - \partial_y^3 \phi) + \mbox{nonlinearity}.\label{eq:toy:w}
\end{align}
Inspired by \cite{MasmoudiWong12b} (see also~\cite{GerardVaretMasmoudi13,KukavicaMasmoudiVicolWong14}), we tackle this issue by considering the {\em linearly-good unknown} 
\begin{align} 
g(t,x,y) = \omega(t,x,y) - u(t,x,y) a(t,y)
\label{eq:g}
\end{align}
where
\begin{align} 
a(t,y) = \frac{\partial_y^2 \phi(t,y)}{\partial_y \phi(t,y)}.
\label{eq:a:def}
\end{align}
Note that one may solve the first order (in $y$) linear equation \eqref{eq:g} to compute $u$ from $g$ explicitly as 
\begin{align} 
u(t,x,y) 
= \partial_y \phi(t,y) \int_0^y g(t,x,\bar y) \frac{1}{\partial_y \phi(t,\bar y)} d \bar y  
= \Phi'(z(t,y)) \int_0^y g(t,x,\bar y) \frac{1}{\Phi'(z(t,\bar y))} d\bar y,
\label{eq:u:g}
\end{align}
where we have used the boundary condition of $u$ at $y=0$. 
Also, if $g$ decays sufficiently fast at infinity, this ensures the correct boundary conditions for $u$.
The formula \eqref{eq:u:g} is useful when performing weighted estimates for $u$ in terms of weighted norms of $g$. 

The evolution equation obeyed by prognostic variable $g$ is
\begin{align} 
\partial_t g - \partial_y^2 g + ( u + \kappa \phi) \partial_x g  + v \partial_y g - 2 \partial_y a g  + u L + v u \partial_y a = \kappa F, \label{eq:New:P:g}
\end{align}
where the diagnostic variables $u$ and $v$ may be computed from $g$ via \eqref{eq:u:g} and \eqref{eq:v:u}. The functions $F$ and $L$ are given by 
\begin{align*} 
F(t,y) &= a (\partial_t \phi - \partial_y^2 \phi) -   ({\partial_t \partial_y\phi} - \partial_y^3 \phi) \\
L(t,y) &= \partial_t a - \partial_y^2 a - 2 a \partial_y a,
\end{align*}
and $a$ is as defined in \eqref{eq:a:def}. 
Moreover, in view of \eqref{eq:New:P:BC:1}, \eqref{eq:New:P:BC:2}, and \eqref{eq:New:P:BC:3}, the linearly-good unknown obeys the boundary conditions
\begin{align} 
(\partial_y g+ag )|_{y=0} = \partial_y \omega|_{y=0} = - \kappa \partial_y^2 \phi|_{y=0},  \label{eq:g:BC:Zero}
\end{align}
and
\begin{align} 
 g|_{y=\infty} = 0, \label{eq:g:BC:Infty}
\end{align}
where the latter one comes from the convenience of vorticity that vanishes as $y \to \infty$.

\subsection{A Gaussian lift of the boundary conditions}
At this stage we make a choice for the boundary condition lift $\Phi$. Our choice is determined by trying to eliminate the forcing term $F$ on the right side of \eqref{eq:New:P:g}, and the linear term $uL$ on the left side of \eqref{eq:New:P:g}.
For this purpose, let $\Phi$ be defined via $\Phi(0) = 0$ and
\begin{align} 
\Phi'(z) =  \frac{1}{\sqrt{\pi}}  \exp\left( - \frac{z^2}{4} \right) 
\label{eq:phi:def}
\end{align}
where the normalization
ensures that $\Phi \to 1$ as $z \to \infty$. In the original variables, this means that
\begin{align*} 
\phi(t,y) = \frac{1}{\sqrt{\pi}} \int_0^{y/\sqrt{\tt}} \exp\left(  - \frac{z^2}{4} \right) dz = \erf\left( \frac{y}{\sqrt{4 \tt}} \right)
\end{align*}
where $\erf$ is the Gauss error function.
With this choice of $\Phi$, 
we immediately obtain
\begin{align*} 
a(t,y) &= - \frac{z}{2 \tt^{1/2}} = - \frac{y}{2 \tt} \\
F(z) &= L(z) = 0 
\end{align*}
and the boundary values
\begin{align*} 
\Phi''(0) = 0 \qquad \mbox{and} \qquad \Phi'(0) = 1/\sqrt{\pi} > 0.
\end{align*}
The evolution equation \eqref{eq:New:P:g} for the good unknown
\begin{align} 
g = \omega + \frac{y}{2\tt} u
\label{eq:g:def}
\end{align}
thus becomes 
\begin{align} 
&\partial_t g - \partial_y^2 g + (u + \kappa \phi) \partial_x g  + v \partial_y g + \frac{1}{\tt} g  - \frac{1}{2\tt} v u  = 0 \label{eq:g:1}\\
&\partial_y g|_{y=0} = g|_{y=\infty} = 0 \label{eq:g:2}\\
&g|_{t=0} = g_0 \label{eq:g:3}.
\end{align}
As noted before (cf.~\eqref{eq:u:g} and \eqref{eq:v:u}), $u$ and $v$ may be computed from $g$ explicitly
\begin{align}  
u(t,x,y) &= \uu(g)(t,x,y) := \exp\left(-\frac{y^2}{4 \tt} \right) \int_0^y g(t,x,\bar y) \exp \left(\frac{\bar y^2}{4 \tt} \right) d\bar y \label{eq:def:uu} \\
v(t,x,y) &= \vv(g)(t,x,y) := - \int_0^y \uu(\partial_x g)(t,x,\bar y) d\bar y ,\label{eq:def:vv} 
\end{align}
and thus solving \eqref{eq:g:1}--\eqref{eq:g:3} is equivalent to solving the Prandtl boundary layer equations \eqref{eq:P:1}--\eqref{eq:P:IC}. 

\subsection{Tangentially analytic functions with Gaussian normal weights}

Lastly, in view of \eqref{eq:u:g} and the choice \eqref{eq:phi:def} of $\Phi$, it is natural to use the Gaussian weight   defined by
\begin{align} 
\ww(t,y) = \exp \left( \frac{\alpha z(t,y)^2}{4} \right) = \exp \left( \frac{\alpha y^2}{4 \tt} \right)
\label{eq:theta:def}
\end{align}
for some 
\begin{align*} 
\alpha \in [1/4,1/2]
\end{align*}
to be chosen later ($\eps$-close to $1/2$).

In order to define the functional spaces in which the solution lies, motivated by~\cite{KukavicaVicol13a}, it is convenient to define
\begin{align*} 
M_m = \frac{\sqrt{m+1}}{m!},
\end{align*}
and introduce the Sobolev weighted semi-norms
\begin{align}
&X_m= X_m(g,\tau) = \|\ww \partial_x^m g\|_{L^2} \tau^m M_m, \label{eq:Xm} \\ 
&D_m = D_m(g,\tau) = \|\ww \partial_y \partial_x^m g\|_{L^2} \tau^m M_m = X_m(\partial_y g,\tau),  \\ 
&Z_m = Z_m(g,\tau) = \|z \ww \partial_x^m g\|_{L^2} \tau^m M_m = X_m(z g,\tau),   \\ 
&B_m = B_m(g,\tau) =  \langle t \rangle^{1/4} X_{m}(g,\tau)+ \langle t \rangle^{1/4} Z_{m}(g,\tau) + \langle t \rangle^{3/4} D_{m}(g,\tau)\\
&Y_m = Y_m(g,\tau) = \|\ww \partial_x^m g\|_{L^2} \tau^{m-1} m M_m. \label{eq:Ym}
\end{align}
As in \cite{KukavicaVicol13a}, we consider the following space of function that are real-analytic in $x$ and lie in a weighted $L^2$ space with respect to $y$
\begin{align*} 
X_{\tau,\alpha} = \{ g(t,x,y) \in L^2(\HH; \theta_\alpha dy dx) \colon \|g\|_{X_{\tau,\alpha}} < \infty \}
\end{align*}
where for $\tau >0$ and $\alpha$ as above we define
\begin{align} 
\|g\|_{X_{\tau,\alpha}} = \sum_{m\geq 0} X_m(g,\tau).
\label{eq:X:tau}
\end{align}
We also define the semi-norm
\begin{align}
\|g\|_{Y_{\tau,\alpha}} = \sum_{m\geq 1} Y_m(g,\tau)
\label{eq:Y:tau}
\end{align}
which encodes the one-derivative gain in the analytic estimates, when the summation in $m$ is considered in $\ell^1$ rather than in $\ell^2$, as is classical when using Fourier analysis. Note that for $\beta >1$, we have
\begin{align}
\|g\|_{Y_{\tau,\alpha}} \leq  \tau^{-1} \|g\|_{X_{\beta \tau,\alpha}}  \sup_{m\geq 1} \left( m \beta^{-m} \right)
\leq C_\beta \tau^{-1} \|g\|_{X_{\beta \tau,\alpha}}
\label{eq:Y:X:bound}.
\end{align}
In particular, $g \in X_{2 \tau,\alpha}$ implies that  $\|g\|_{Y_{\tau,\alpha}} \leq  \tau^{-1} \|g\|_{X_{2\tau,\alpha}}$.
The gain of a $y$ derivative shall be encoded in the dissipative semi-norm
\begin{align*} 
\|g\|_{D_{\tau,\alpha}} = \sum_{m\geq 0} D_m(g,\tau) = \| \partial_y g \|_{X_{\tau,\alpha}},
\end{align*}
while the damping in the heat self-similar variable $z$ is measured via
\begin{align*} 
\|g\|_{Z_{\tau,\alpha}} = \sum_{m\geq 0} Z_m(g,\tau) = \| z g \|_{X_{\tau,\alpha}}.
\end{align*}
For compactness of notation, for a function $g$ such that $g, z g, \partial_y g \in X_{\tau,\alpha}$
we use the time-weighted norm
\begin{align}
\|g\|_{B_{\tau,\alpha}} = \sum_{m\geq 0} B_m(g,\tau) = \tt^{1/4} \|g\|_{X_{\tau,\alpha}} + \tt^{1/4} \|g\|_{Z_{\tau,\alpha}} + \tt^{3/4} \|g\|_{D_{\tau,\alpha}} 
\label{eq:B:tau}
\end{align}
where as before $\tau > 0$ and $\alpha \in [1/4,1/2]$.
Lastly, in order to obtain time regularity for the radius of analyticity $\tau(t)$, it will be convenient to use a hybrid of the $\ell^2$ and $\ell^1$ tangentially analytic norms, given by
\begin{alignat}{2}
&\|g\|_{\tilde D_{\tau,\alpha}}:= \sum_{m\geq 0} \tilde D_m, \quad
& &\tilde D_m = \tilde D_m(g,\tau)  = \frac{ D_{m}(g,\tau)^2}{X_m(g,\tau)}, \label{eq:tilde:D:tau}\\ 
&\|g\|_{\tilde Z_{\tau,\alpha}}:= \sum_{m\geq 0} \tilde Z_m,  \quad
& & \tilde Z_m = \tilde Z_m (g,\tau) = \frac{Z_{m}(g,\tau)^2}{X_m(g,\tau)},\\ 
&\|g\|_{\tilde B_{\tau,\alpha}}:= \sum_{m\geq 0} \tilde B_m \quad 
& & \tilde B_m = \tilde B_{m}(g,\tau)= \langle t \rangle^{1/4} X_{m}(g,\tau)+ \langle t \rangle^{1/4}  \tilde Z_{m}(g,\tau) + \langle t \rangle^{5/4} \tilde D_m(g,\tau).\label{eq:tilde:B:tau}
\end{alignat}
We note that the bound
\begin{align} 
\|g\|_{B_{\tau,\alpha}} \leq 2 \tt^{1/8} \|g\|_{X_{\tau,\alpha}}^{1/2}  \|g\|_{\tilde B_{\tau,\alpha}}^{1/2}
\label{eq:B:tilde:B}
\end{align}
is an immediate consequence of the Cauchy-Schwartz inequality.

\subsection{The main result} Having introduced the functional setting of this paper we restate Theorem~\ref{thm:main} in these terms. First, we give a definition of solutions to the reformulated Prandtl equations~\eqref{eq:g:1}--\eqref{eq:g:3}. 
\begin{definition}[\bf Classical in $x$ weak in $y$ solutions]
\label{def:sol}
For $\beta >0$ define $\HCal_{2,1,\beta}$ to be the closure under the norm 
\[
\| h \|_{\HCal_{2,1,\beta}}^2 = \sum_{m=0}^{2} \sum_{j=0}^1 \int_{\HH} |\partial_x^m \partial_y^j h(x,y)|^2 \exp\left(\frac{\beta y^2}{2}\right) dy dx
\]
of the set of functions 
\[
{\mathcal D} = \{ h(x,y) \in C_0^\infty(\RR \times [0,\infty)) \colon \partial_y h|_{y=0} = 0 \}.
\]
Let $\alpha \in [1/4,1/2]$, and $\ww(t,y)$ be defined by \eqref{eq:theta:def}. For $T>0$ we say that a function
\begin{align*} 
g  \in L^\infty([0,T); \HCal_{2,1,\alpha/\tt}) 
\end{align*}
is a {\em classical in $x$ weak in $y$ solution} of the initial value problem for the Prandtl equations \eqref{eq:g:1}--\eqref{eq:g:3} on $[0,T)$, if \eqref{eq:g:1} holds when tested against elements of $C_0^\infty([0,T) \times \RR \times [0,\infty))$.
\end{definition}

\begin{theorem}[\bf Main result]
\label{thm:rigorous}
Assume the trace of the Euler flow is given by $u^E = \kappa$ and $\partial_x p^E = 0$.
For $t\geq 0$, define
\begin{align*} 
u(t,x,y) = u^P(t,x,y) - \kappa\, \erf\left( \frac{y}{\sqrt{4 \tt }}  \right)
\end{align*}
and  let 
\[
g(t,x,y) = \partial_y u (t,x,y) + \frac{y}{2\tt} u(t,x,y).
\]
There exists a sufficiently large universal constant $C_*>0$ and a sufficiently small universal constant $\eps_*>0$ such that the following holds. Assume that there exists an analyticity radius $\tau_0>0$ and an $\eps \in (0,\eps_*]$ 
such that 
\begin{align} 
\frac{C_*}{\log \frac 1 \eps} \leq \tau_0^{3/2} \leq \frac{1}{C_* \eps^3},
\label{eq:eps:tau:0}
\end{align}
and such that the initial condition $g_0=g(0,\cdot,\cdot)$ is small, in the sense that 
\begin{align} 
\|g_0\|_{X_{2 \tau_0,1/2}} \leq \eps.
\label{eq:cond:1}
\end{align}
Then there exists a unique classical in $x$ weak in $y$ solution 
$g$ of the Prandtl boundary layer equations \eqref{eq:g:1}--\eqref{eq:g:3} on $[0,T_\eps]$ which is tangentially real-analytic, and the maximal time of existence obeys
\begin{align*} 
T_\eps \geq \exp\left(\frac{\eps^{-1}}{\log (\eps^{-1})}\right).
\end{align*}
Moreover, letting $\delta = \eps \log \frac 1\eps$ and  $\alpha = \frac{1 - \delta}{2}$, the tangential analyticity radius $\tau(t)$ of the solution $g(t)$ satisfies 
\begin{align}
\tau(t) \geq \left( \tau_0^{3/2} - \frac{C_* \tt^\delta}{2 \log \frac 1 \eps} \right)^{2/3} \geq \frac{\tau_0}{2}
\label{eq:g:bnd:0}
\end{align}
and the solution $g(t)$ obeys the bounds
\begin{align} 
& \|g(t)\|_{X_{\tau(t),\alpha}} \leq \eps \tt^{- 5/4 + \delta}  \label{eq:g:bnd:1} \\
& \int_0^t \left( \|g(s)\|_{B_{\tau(s),\alpha}} + \|g(s)\|_{\tilde B_{\tau(s),\alpha}} \right) ds \leq \frac{C_* \tt^\delta}{\log \frac 1\eps} \label{eq:g:bnd:2} \\
& \int_0^t \frac{\ss^{5/4-\delta}}{\tau(s)^{1/2}} \|g(s)\|_{Y_{\tau(s),\alpha}} \|g(s)\|_{B_{\tau(s),\alpha}} ds \leq C_* \eps \label{eq:g:bnd:3}
\end{align}
for all $t \in [0,T_\eps]$.
\end{theorem}
It follows from the estimates in the next section (cf.~Lemmas~\ref{lem:u:v:g} and~\ref{lem:analytic}) that bounds on $g$, $z g$, and $\partial_y g$ in  $X_{\tau,\alpha}$ imply similar bounds on $u$ and $v$ in $X_{\tau,\alpha}$, and thus \eqref{eq:g:bnd:1}--\eqref{eq:g:bnd:3} directly translate into  bounds for $u^P$ and $v^P$. Moreover, when $g(t) \in \HCal_{2,1,\alpha/\tt}$, then $u(t)$  lies in $\HCal_{2,2,\alpha/\tt}$ and the Prandtl equations \eqref{eq:P:1} hold pointwise in $x$ and in an $L^2$ sense in $y$. We omit these details. 
The proof of Theorem~\ref{thm:rigorous} consists of a priori estimates (cf.~Section~\ref{sec:apriori}), the proof of uniqueness of solutions in this class (cf.~Section~\ref{sec:uniqueness}), and the construction of solutions (cf.~Section~\ref{sec:existence}).

\section{A priori estimates}
\label{sec:apriori}
In this section we give the a priori estimates needed to prove Theorem~\ref{thm:rigorous}. We start with a number of preliminary lemmas, which lead up to Subsection~\ref{sec:conclusion}, where we conclude the a priori bounds.
\subsection{Bounding the diagnostic variables in terms of the prognostic one}
We may use \eqref{eq:u:g} to write
\begin{align} 
\ww(y) u(y) =  \int_0^y  g(\bar y) \ww(\bar y)  \exp\left(\frac{(1-\alpha)}{4 \tt} (\bar y^2 - y^2)\right) d\bar y.
\label{eq:theta:u}
\end{align}
On the one hand, it is immediate from the above that
\begin{align}
\|\ww u\|_{L^\infty_y} \leq \|\ww g\|_{L^1_y}. \label{eq:u:infty:1}
\end{align}
On the other hand,  for $p \in [1,2]$ we may estimate
\begin{align*}
|\ww(y) u(y)| 
&\leq \|\ww g\|_{L^{p/(p-1)}} \left( \int_0^y \exp\left(\frac{p(1-\alpha)}{4 \tt} (\bar y^2 -  y^2)\right) d\bar y \right)^{1/p} \notag\\
&= \| \ww g\|_{L^{p/(p-1)}} \tt^{1/(2p)} \left( \frac{\DWF[z(t,y) K_{p,\alpha}]}{K_{p,\alpha}} \right)^{1/p} 
\end{align*}
where $K_{p,\alpha} =  \sqrt{p(1-\alpha)}/2$
and
\begin{align*}
\DWF[y] = \exp(-y^2) \int_0^y \exp(\bar y^2) d\bar y = \int_0^y \exp(\bar y^2-y^2) d\bar y . 
\end{align*}
It is not hard to check that
\begin{align*}
\DWF[y] \leq \frac{2}{1+y} 
\end{align*}
for all $y\geq 0$. Because there exists a universal constant $C>0$ such that $1/C \leq K_{p,\alpha} \leq C$
for $p \in [1,2]$ and $\alpha \in [1/4,1/2]$, 
it follows that 
\begin{align}
|\ww(t,y) u (t,x,y) | \leq C \tt^{1/(2p)} \|\ww g\|_{L^{p/(p-1)}_y} \frac{1}{(1+z(t,y))^{1/p}}
\label{eq:u:BOUND}
\end{align}
for $p\in [1,2]$. Using \eqref{eq:u:BOUND} and recalling the definition of $v$ in \eqref{eq:v:u} we may prove the following estimates.

\begin{lemma}[\bf Bounds for the diagnostic variables]
\label{lem:u:v:g}
Let $\ww$ be given by \eqref{eq:theta:def} with $\alpha\in [1/4,1/2]$, Define $u = \uu(g)$ and $v = \vv(g)$  by \eqref{eq:def:uu} respectively \eqref{eq:def:vv}. For $m \geq 0$ we have
\begin{align}
\|\partial_x^m u \|_{L^2_x L^\infty_y} &\leq C \tt^{1/4} \|\ww \partial_x^m g\|_{L^2_{x,y}} \label{eq:uu:2:infty}\\
\|\partial_x^m u \|_{L^\infty_{x,y}} &\leq C \tt^{1/4} \|\ww \partial_x^m g\|_{L^2_{x,y}}^{1/2} \|\ww \partial_x^{m+1} g\|_{L^2_{x,y}}^{1/2} \label{eq:uu:infty:infty}\\
\|\ww \partial_x^m g\|_{L^1_y} &\leq C \tt^{1/4} \|\ww \partial_x^m g\|_{L^2_y}^{1/2} \| z \ww \partial_x^m g\|_{L^2_y}^{1/2} \label{eq:g:2:1}\\
\|\ww \partial_x^m u\|_{L^2_{x,y}} &\leq C \tt^{3/4} \|\ww \partial_x^m g\|_{L^2_{x,y}}^{1/2}  \|\ww \partial_x^m \partial_y g\|_{L^2_{x,y}}^{1/2}  + C \tt^{1/2} \|  \ww \partial_x^m g\|_{L^2_{x,y}}^{1/2}    \|  z \ww \partial_x^m g\|_{L^2_{x,y}}^{1/2}
 \label{eq:u:2:2} \\
\|\ww \partial_x^m u\|_{L^\infty_x L^2_{y}} &\leq C \tt^{3/4} \|\ww \partial_x^m g\|_{L^2_{x,y}}^{1/4} \|\ww \partial_x^{m+1} g\|_{L^2_{x,y}}^{1/4}   \|\ww \partial_x^m \partial_y g\|_{L^2_{x,y}}^{1/4}  \|\ww \partial_x^{m+1} \partial_y g\|_{L^2_{x,y}}^{1/4} \notag\\
&\quad + C \tt^{1/2} \|  \ww \partial_x^m g\|_{L^2_{x,y}}^{1/4} \| \ww \partial_x^{m+1} g\|_{L^2_{x,y}}^{1/4}   \| z \ww \partial_x^m g\|_{L^2_{x,y}}^{1/4} \|  z \ww \partial_x^{m+1} g\|_{L^2_{x,y}}^{1/4} \label{eq:u:infty:2} \\
\|\partial_x^m v \|_{L^2_x L^\infty_y} &\leq  C \tt^{3/4}\|\ww \partial_x^{m+1}g\|_{L_{x,y}^2}  \label{eq:v:2:infty}\\
\|\partial_x^m v \|_{L^\infty_{x,y}} &\leq C \tt^{3/4} \|\ww \partial_x^{m+1}g\|_{L_{x,y}^2}^{1/2} \|\ww \partial_x^{m+2}g\|_{L_{x,y}^2}^{1/2} \label{eq:v:infty:infty}
\end{align}
for some universal constant $C>0$, which is independent of $\alpha \in [1/4,1/2]$.
\end{lemma}
\begin{proof}[Proof of Lemma~\ref{lem:u:v:g}]
From identity \eqref{eq:theta:u} we have
\begin{align*} 
|\partial_x^m u(y)|
&\leq \frac{1}{\ww(y)} \int_0^y |\ww(\bar y) \partial_x^m g(\bar y) | \exp\left(\frac{(1-\alpha)}{4 \tt} (\bar y^2 - y^2)\right) d\bar y\notag\\
&\leq \|\ww \partial_x^m g\|_{L^2_y} \frac{\sqrt{y}}{\ww(y)} = \tt^{1/4} \|\ww \partial_x^m g\|_{L^2_y} \sqrt{z} \exp\left(- \frac{\alpha z^2}{4}\right) \notag\\
&\leq C \tt^{1/4} \|\ww \partial_x^mg\|_{L^2_y}.
\end{align*}
The bound \eqref{eq:uu:2:infty} follows by taking the $L^2$ norm in $x$ of the above, while the bound \eqref{eq:uu:infty:infty} follows upon additionally applying the 1D Agmon inequality in the $x$ variable,
\begin{align*}
\|f \|_{L^\infty_x} \leq C \|f\|_{L^2_x}^{1/2} \|\partial_x f\|_{L^2_x}^{1/2}.
\end{align*}
To bound $\ww \partial_x^m u$, we note that for  $R>0$ we have
\begin{align*}
\| \ww \partial_x^m g\|_{L^1_y} 
&= \int_0^R |\ww (y) \partial_x^m g(y)| dy + \int_R^\infty |y \ww(y) \partial_x^m g(y)| |y|^{-1} dy\notag\\
&\leq R^{1/2} \|\ww \partial_x^m g\|_{L^2_y} +R^{-1/2} \| y \ww \partial_x^m g\|_{L^2_y}
\end{align*}
which upon optimizing in $R$ yields
\begin{align*}
\|\ww \partial_x^m g\|_{L^1_y} 
&\leq C \|\ww \partial_x^m g\|_{L^2_y}^{1/2} \| y \ww \partial_x^m g\|_{L^2_y}^{1/2} \notag\\
&\leq C \tt^{1/4} \|\ww \partial_x^m g\|_{L^2_y}^{1/2} \| z \ww \partial_x^m g\|_{L^2_y}^{1/2}.
\end{align*}
Upon taking $L^2$ norm in $x$, this proves \eqref{eq:g:2:1}.
When combined with \eqref{eq:u:infty:1}, we obtain from the above that
\begin{align*}
\|\ww \partial_x^m u\|_{L^\infty_y} \leq C  \tt^{1/4} \|\ww \partial_x^m g\|_{L^2_y}^{1/2} \| z \ww \partial_x^m g\|_{L^2_y}^{1/2}.
\end{align*}
In order to prove \eqref{eq:u:2:2}, we use \eqref{eq:u:BOUND} with $p=1$ and the 1D Agmon inequality
in the $y$ variable to obtain
\begin{align*}
\|\ww \partial_x^m u\|_{L^2_y} 
&\leq C \tt^{1/2} \|\ww \partial_x^m g\|_{L^\infty_y} \| (1+z(t,y))^{-1}\|_{L^2_y}\notag\\
&\leq C \tt^{3/4} \|\ww \partial_x^m g\|_{L^2_y}^{1/2} \left( \|\ww \partial_x^m \partial_y g\|_{L^2_y}  + \|  \partial_y \ww \partial_x^m g\|_{L^2_y}  \right)^{1/2} \notag\\
&\leq C \tt^{3/4} \|\ww \partial_x^m g\|_{L^2_y}^{1/2}  \|\ww \partial_x^m \partial_y g\|_{L^2_y}^{1/2}  + C \tt^{1/2} \| \ww \partial_x^m g\|_{L^2_y}^{1/2}    \|z \ww \partial_x^m g\|_{L^2_y}^{1/2}.
\end{align*}
Taking the $L^2$ norm in $x$ of the above yields \eqref{eq:u:2:2}, while an application of the 1D Agmon inequality in $x$ gives \eqref{eq:u:infty:2}.
For the $v$ bounds, we use \eqref{eq:u:BOUND} with $p=2$ and obtain
\begin{align*} 
\|\partial_x^m v\|_{L^\infty_y} 
\leq \|\partial_x^{m+1} u\|_{L^1_y} 
&\leq \|\ww \partial_x^{m+1} u\|_{L^\infty_y} \|\ww^{-1}\|_{L^1_y}  \notag\\
&\leq C \tt^{1/2} \|\ww \partial_x^{m+1} u\|_{L^\infty_y}  \notag\\
&\leq C \tt^{3/4} \| \ww \partial_x^{m+1} g\|_{L^2_y}.
\end{align*}
Integrating in $x$ the above implies \eqref{eq:v:2:infty}. An extra use of the 1D Agmon inequality yields \eqref{eq:v:infty:infty}.
\end{proof}

\begin{remark} The first two estimates in Lemma~\ref{lem:u:v:g} also hold in the case when we don't use the weight (i.e., when $\ww=1$). Indeed, we use the relation
\begin{align*}
u(y)=\int_0^yg(\bar{y})\exp\left(\frac{\bar{y}^2-y^2}{4\tt}\right)\,d\bar{y},
\end{align*}
which implies that
\begin{align*}
|u(y)|\le C\tt^{1/2p}\|g\|_{L_y^{p/p-1}}\frac{1}{(1+z)^{1/p}}
\end{align*}
for $1\le p\le\infty$.
\end{remark}

\subsection{Weighted Sobolev energy estimates for the good unknown} 

Let $m\geq 0$. We apply $\partial_x^m$ to
\eqref{eq:g:1}, multiply the resulting equation with $\ww^2\partial_x^mg$ and integrate over $\HH$ to obtain
\begin{align}
\frac{1}{2}&\frac{d}{dt}\|\ww\partial_x^mg\|_{L^2}^2
+\|\ww\partial_y\partial_x^mg\|_{L^2}^2
+\frac{\alpha(1-2\alpha)}{4\tt}\|z\ww \partial_x^mg\|_{L^2}^2
+\frac{2-\alpha}{2\tt}\|\ww\partial_x^mg\|_{L^2}^2\notag\\
&=-\sum_{j=0}^m\binom{m}{j}\int\partial_x^{m-j}u\ww\partial_x^{j+1}g\ww\partial_x^mg
-\sum_{j=0}^m\binom{m}{j}\int\partial_x^{m-j}v\ww\partial_y\partial_x^jg\ww\partial_x^mg \notag\\
&\qquad +\frac{1}{2\tt}\sum_{j=0}^m\binom{m}{j}\int\partial_x^jv\ww\partial_x^{m-j}u\ww\partial_x^mg\notag\\
&=U_m+V_m+T_m.
\label{eq:g:m}
\end{align} 
Here we have used the boundary conditions \eqref{eq:g:2}--\eqref{eq:g:3} and the cancellation
\begin{align*} 
\int \phi \partial_x^{m+1} g \ww^2 \partial_x^m g dx  = 0,
\end{align*}
which follows upon integration by parts and the fact that $\partial_x (\phi \ww^2) = 0$. Dividing \eqref{eq:g:m} by $\|\ww\partial_x^mg\|_{L^2}$, multiplying by $\tau^m M_m$, and using the notations \eqref{eq:Xm}--\eqref{eq:Ym} and \eqref{eq:tilde:D:tau}--\eqref{eq:tilde:B:tau}, we arrive at 
\begin{align}
\frac{d}{dt} X_m
+\tilde D_m
+\frac{\alpha(1-2\alpha)}{4\tt} \tilde Z_m
+\frac{2-\alpha}{2\tt} X_m
= \frac{\tau^m M_m}{\|\theta_\alpha \partial_x^m g\|_{L^2}} (U_m + V_m+ T_m)
\label{eq:g:m:*}
\end{align}
In the next subsection we obtain lower bounds for the dissipative and damping terms on the left side of \eqref{eq:g:m:*}, while in the following subsection we estimate the nonlinear terms on the right side of \eqref{eq:g:m:*}.

\subsection{Bounds for the dissipative and damping terms}
\begin{lemma}[\bf Poincar\'e inequality with gaussian weights]
\label{lem:Carleman}
Let $g$ be such that $\partial_y g|_{y=0} = 0$ and $g|_{y=\infty} = 0$. For $\alpha \in [1/4,1/2]$, $m \geq 0$, and $t\geq 0$ it holds that
\begin{align}
 \frac{\alpha}{\tt} \| \ww \partial_x^m g\|_{L^2_y}^2 \leq \|\ww \partial_y \partial_x^m g\|_{L^2_y}^2 
\label{eq:CE}
\end{align}
where $\ww(t,y)=\exp\left(\frac{\alpha y^2}{4\tt}\right)$. 
\end{lemma}
\begin{proof}[Proof of Lemma~\ref{lem:Carleman}]
The above inequality is classical, and it is a special case of the Treves inequality which can be found in~\cite{Hormander83}. For simplicity, we give a short proof for the case $m=0$. 
Note that 
\[
\ww \partial_y g = \partial_y (\ww g) - \frac{\alpha y}{2\tt} \ww g
\]
as can be checked directly. Using that $(a-b)^2 = (a+b)^2 - 4 ab$ it then follows that 
\begin{align*}
\int (\ww \partial_y g )^2 \,dy
&=\int \left(\partial_y  (\ww g  ) -\frac{\alpha y}{2\tt} \ww g\right)^2\,dy\\
&=\int \left(\partial_y  (\ww g  ) + \frac{\alpha y}{2\tt} \ww g\right)^2\,dy - \frac{\alpha}{\tt} \int 2 y (\ww g) \partial_y (\ww g) \, dy\\
&=\int \left(\partial_y  (\ww g  ) + \frac{\alpha y}{2\tt} \ww g\right)^2\,dy + \frac{\alpha}{\tt} \int   (\ww g)^2 \, dy\\
&\ge\frac{\alpha}{\tt}\int (\ww g)^2 \,dy,
\end{align*}
upon integrating by parts with respect to $y$ in the third equality. No boundary terms arise in this process.
\end{proof}

Using Lemma~\ref{lem:Carleman} we may bound the dissipation term in \eqref{eq:g:m} from below as
\begin{align}
\frac{\|\ww \partial_y \partial_x^m g\|_{L^2}^2}{\|\ww \partial_x^m g\|_{L^2}} 
&\geq \frac{\beta}{2}\frac{\|\ww \partial_y \partial_x^m g\|_{L^2}^2}{\|\ww \partial_x^m g\|_{L^2}}  + \frac{2-\beta}{2} \frac{\alpha^{1/2}}{\tt^{1/2}} \|\ww \partial_y \partial_x^m g\|_{L^2} \notag\\
&\geq  \frac{\beta}{2}\frac{\|\ww \partial_y \partial_x^m g\|_{L^2}^2}{\|\ww \partial_x^m g\|_{L^2}} + \frac{\alpha^{1/2} \beta}{2 \tt^{1/2}} \|\ww \partial_y \partial_x^m g\|_{L^2} +  \frac{\alpha (1-\beta)}{\tt} \|\ww\partial_x^m g\|_{L^2} 
\label{eq:g:m:D}
\end{align}
where $\beta \in (0,1/2)$ is to be chosen precise later.

For the damping terms in \eqref{eq:g:m} we have the lower bounds
\begin{align} 
&\frac{1}{\|\ww \partial_x^m g\|_{L^2}} \left( \frac{\alpha(1-2\alpha)}{4\tt}\|z\ww \partial_x^mg\|_{L^2}^2
+\frac{2-\alpha}{2\tt}\|\ww\partial_x^mg\|_{L^2}^2 \right) \notag\\
&=\frac{1}{\|\ww \partial_x^m g\|_{L^2}} \left( \frac{\alpha}{4\tt}\|((1-2\alpha)z^2+ 4 \gamma)^{1/2} \ww \partial_x^mg\|_{L^2}^2
+\frac{1-\alpha/2-\alpha \gamma}{\tt}\|\ww\partial_x^mg\|_{L^2}^2 \right) \notag\\
& \geq \frac{\alpha (1-2\alpha)}{8\tt} \frac{\|z \ww \partial_x^m g\|_{L^2}^2}{\|\ww \partial_x^m g\|_{L^2}}  + \frac{\alpha\gamma^{1/2}(1-2\alpha)^{1/2}}{4\tt} \|z \ww \partial_x^m g\|_{L^2} +  \frac{1-\alpha/2-\alpha \gamma}{\tt}\|\ww\partial_x^mg\|_{L^2}  \label{eq:g:m:Z}.
\end{align}
In the last inequality above we used that 
\begin{align*}
((1-2\alpha)z^2+ 4 \gamma)^{1/2} & \geq 2 \gamma^{1/2}\notag\\
((1-2\alpha)z^2+ 4 \gamma)^{1/2}& \geq (1-2\alpha)^{1/2} z
\end{align*}
which holds for all $z \geq 0$, when $\alpha \in [1/4,1/2]$ and $\gamma \in [0,1/2]$.
In summary, in this subsection we have proven the following bounds.
\begin{lemma}[\bf Lower bounds for the damping and dissipative terms]
\label{lem:LEFT}
Fix $\alpha \in [1/4,1/2]$, and let $\beta,\gamma \in [0,1/2]$ be arbitrary. Then we have
\begin{align}
&\sum_{m\geq 0} \left( \tilde D_m +  \frac{\alpha(1-2\alpha)}{4\tt} \tilde Z_m +\frac{2-\alpha}{2\tt} X_m \right) \notag\\
&\quad \geq \frac{\beta}{2} \|g\|_{\tilde D_{\tau,\alpha}} + \frac{\alpha(1-2\alpha)}{8\tt} \|g\|_{\tilde Z_{\tau,\alpha}}
+ \frac{\alpha^{1/2} \beta }{2 \tt^{1/2}} \|g\|_{D_{\tau,\alpha}} + \frac{\alpha\gamma^{1/2}(1-2\alpha)^{1/2}}{4\tt} \|g\|_{Z_{\tau,\alpha}}\notag\\
&\qquad + \frac{1 + \alpha(1/2-\gamma-\beta)}{\tt} \|g\|_{X_{\tau,\alpha}}
\label{eq:LEFT}
\end{align}
independently of $\tau>0$.
\end{lemma}
\begin{proof} The lemma follows upon recasting \eqref{eq:g:m:D} and \eqref{eq:g:m:Z} as 
\begin{align*} 
\tilde D_m 
&\geq \frac{\beta}{2} \tilde D_m + \frac{\alpha^{1/2} \beta }{2 \tt^{1/2}} D_m + \frac{\alpha (1-\beta)}{\tt} X_m \\
\frac{2-\alpha}{2\tt} X_m + \frac{\alpha(1-2\alpha)}{4\tt} \tilde Z_m 
&\geq \frac{\alpha(1-2\alpha)}{8\tt} \tilde Z_m + \frac{\alpha\gamma^{1/2}(1-2\alpha)^{1/2}}{4\tt}  Z_m +  \frac{1-\alpha/2-\alpha \gamma}{\tt} X_m
\end{align*}
and summing over $m\geq 0$.
\end{proof}

\subsection{Bounds for the nonlinear terms}
In this subsection we bound the nonlinear terms on the right side of \eqref{eq:g:m:*} for every $m\geq 0$, cf.~estimates~\eqref{eq:g:m:U}, \eqref{eq:g:m:V}, \eqref{eq:g:m:T} below. When summed over $m\geq 0$ we obtain the following tangentially analytic estimates  for the nonlinear terms.
\begin{lemma}[\bf Estimates for the nonlinearity]
\label{lem:analytic}
There exits a universal constant $C\geq 1$ such that the bounds
\begin{align} 
&\sum_{m\geq 0}\frac{|U_m| \tau^m M_m}{\|\ww\partial_x^mg\|_{L^2}} \leq \frac{C \tt^{1/4}}{\tau(t)^{1/2}} \|g\|_{X_{\tau,\alpha}} \|g\|_{Y_{\tau,\alpha}} \\
&\sum_{m\geq 0}\frac{|V_m| \tau^m M_m}{\|\ww\partial_x^mg\|_{L^2}} \leq \frac{C \tt^{3/4}}{\tau(t)^{1/2}} \|g\|_{D_{\tau,\alpha}} \|g\|_{Y_{\tau,\alpha}} \\
&\sum_{m\geq 0}\frac{|T_m| \tau^m M_m}{\|\ww\partial_x^mg\|_{L^2}} \leq \frac{C_0 \tt^{1/4}}{\tau(t)^{1/2}} \|g\|_{X_{\tau,\alpha}}^{1/2} \|g\|_{Z_{\tau,\alpha}}^{1/2} \|g\|_{Y_{\tau,\alpha}} + \frac{C \tt^{1/2}}{\tau(t)^{1/2}} \|g\|_{X_{\tau,\alpha}}^{1/2} \|g\|_{D_{\tau,\alpha}}^{1/2} \|g\|_{Y_{\tau,\alpha}} 
\end{align}
hold for every $\tau>0$ and $\alpha \in [1/4,1/2]$. 
\end{lemma}

\begin{proof}
First, using \eqref{eq:uu:2:infty}--\eqref{eq:uu:infty:infty}, and the 1D Agmon inequality in the $x$ variable we obtain
\begin{align*}
\frac{|U_m|}{\|\ww\partial_x^mg\|_{L^2}}
&\leq   \sum_{j=0}^{[m/2]} \binom{m}{j} \|\partial_x^{m-j}u\|_{L_x^2L_y^{\infty}}\|\ww\partial_x^{j+1}g\|_{L_x^{\infty}L_y^2}  +\sum_{j=[m/2]+1}^{m} \binom{m}{j} \|\partial_x^{m-j}u\|_{L_{x,y}^{\infty}}\|\ww\partial_x^{j+1}g\|_{L_{x,y}^2} \notag\\
&\leq C\tt^{1/4} \sum_{j=0}^{[m/2]}\binom{m}{j} \|\ww\partial_x^{m-j}g\|_{L^2} 
\|\ww\partial_x^{j+1}g\|_{L^2}^{1/2}\|\ww\partial_x^{j+2}g\|_{L^2}^{1/2} \notag\\
&\ \ +C\tt^{1/4}  \sum_{j=[m/2]+1}^m\binom{m}{j}  \|\ww \partial_x^{m-j} g\|_{L^2}^{1/2} \|\ww \partial_x^{m-j+1} g\|_{L^2}^{1/2}  
\|\ww\partial_x^{j+1}g\|_{L^2}
\end{align*}
where $C>0$ is independent of $\alpha \in [1/4,1/2]$. Upon multiplying by $\tau^m M_m$ and using the definitions \eqref{eq:Xm}--\eqref{eq:Ym}, the above bound implies
\begin{align} 
\frac{|U_m| \tau^m M_m}{\|\ww \partial_x^m g\|_{L^2}} \leq  \frac{C\tt^{1/4}}{(\tau(t))^{1/2}} \left( \sum_{j=0}^{[m/2]}  X_{m-j} Y_{j+1}^{1/2} Y_{j+2}^{1/2}    
+ \sum_{j=[m/2]+1}^{m} X_{m-j}^{1/2} X_{m-j+1}^{1/2} Y_{j+1} \right).
\label{eq:g:m:U}
\end{align}

Similarly, by appealing to \eqref{eq:v:2:infty} and \eqref{eq:v:infty:infty} we have
\begin{align*}
\frac{|V_m|}{\|\ww\partial_x^mg\|_{L^2}}
&\leq \sum_{j=0}^{[m/2]} \binom{m}{j} \|\partial_x^{m-j}v\|_{L_x^2L_y^{\infty}}\|\ww\partial_y\partial_x^jg\|_{L_x^{\infty}L_y^2}  +\sum_{j=[m/2]+1}^{m} \binom{m}{j} \|\partial_x^{m-j}v\|_{L_{x,y}^{\infty}}\|\ww\partial_y\partial_x^jg\|_{L_{x,y}^2} \notag\\
&\leq C\tt^{3/4} \sum_{j=0}^{[m/2]}\binom{m}{j}\|\ww\partial_x^{m-j+1}g\|_{L^2}
\|\ww\partial_y\partial_x^jg\|_{L^2}^{1/2}\|\ww\partial_y\partial_x^{j+1}g\|_{L^2}^{1/2} \notag\\
&\ \ 
+C\tt^{3/4} \sum_{j=[m/2]+1}^m \binom{m}{j} \|\ww\partial_x^{m-j+1}g\|_{L^2}^{1/2}\|\ww\partial_x^{m-j+2}g\|_{L^2}^{1/2}
\|\ww\partial_y\partial_x^jg\|_{L^2} 
\end{align*}
where $C>0$ is independent of $\alpha \in [1/4,1/2]$. Upon multiplying by $\tau^m M_m$ and using the definitions \eqref{eq:Xm}--\eqref{eq:Ym}, the above bound implies
\begin{align} 
\frac{|V_m| \tau^m M_m}{\|\ww \partial_x^m g\|_{L^2}} \leq  \frac{C\tt^{3/4}}{(\tau(t))^{1/2}} \left( \sum_{j=0}^{[m/2]} Y_{m-j+1} D_j^{1/2} D_{j+1}^{1/2} 
+ \sum_{j=[m/2]+1}^m Y_{m-j+1}^{1/2} Y_{m-j+2}^{1/2} D_j  \right).
\label{eq:g:m:V}
\end{align}

For the last term on the right of  \eqref{eq:g:m} we appeal to \eqref{eq:u:2:2}, \eqref{eq:u:infty:2}, \eqref{eq:v:2:infty}, and \eqref{eq:v:infty:infty} to obtain
\begin{align*}
\frac{|T_m|}{\|\ww\partial_x^mg\|_{L^2}}
&\leq\frac{1}{2\tt}  \sum_{j=0}^{[m/2]} \binom{m}{j} \|\partial_x^jv\|_{L_{x,y}^{\infty}}\|\ww\partial_x^{m-j}u\|_{L_{x,y}^2} \notag\\
&\ \ +\frac{1}{2\tt}\sum_{j=[m/2]+1}^{m} \binom{m}{j} \|\partial_x^jv\|_{L_x^2L_y^{\infty}}\|\ww\partial_x^{m-j}u\|_{L_x^{\infty}L_y^2} \notag\\
&\leq  C\tt^{1/2} \sum_{j=0}^{[m/2]}\binom{m}{j}
\|\ww \partial_x^{j+1}g\|_{L_y^2}^{1/2}\| \ww\partial_x^{j+2}g\|_{L_y^2}^{1/2} 
\|\ww\partial_x^{m-j}g\|_{L^2}^{1/2}\| \ww \partial_y \partial_x^{m-j}g\|_{L^2}^{1/2} \notag\\
&\ \ + C\tt^{1/4} \sum_{j=0}^{[m/2]}\binom{m}{j}
\|\ww \partial_x^{j+1}g\|_{L_y^2}^{1/2}\| \ww\partial_x^{j+2}g\|_{L_y^2}^{1/2} 
\|\ww\partial_x^{m-j}g\|_{L^2}^{1/2}\|z \ww\partial_x^{m-j} g\|_{L^2}^{1/2} \notag\\
&\ \ +C\tt^{1/2}  \sum_{j=[m/2]+1}^m \binom{m}{j}
\|\ww\partial_x^{j+1}g\|_{L^2} 
\|\ww\partial_x^{m-j}g\|_{L^2}^{1/4}
\|\ww\partial_x^{m-j+1}g\|_{L^2}^{1/4} \notag\\
&\qquad \qquad \qquad \qquad \times
\|\ww\partial_x^{m-j}\partial_y g\|_{L^2}^{1/4}
\|\ww\partial_x^{m-j+1}\partial_y  g\|_{L^2}^{1/4}
\notag\\
&\ \ +C\tt^{1/4}  \sum_{j=[m/2]+1}^m \binom{m}{j}
\|\ww\partial_x^{j+1}g\|_{L^2} 
\|\ww\partial_x^{m-j}g\|_{L^2}^{1/4}
\|\ww\partial_x^{m-j+1}g\|_{L^2}^{1/4}\notag\\
&\qquad \qquad \qquad \qquad \times
\|z\ww\partial_x^{m-j}g\|_{L^2}^{1/4}
\|z\ww\partial_x^{m-j+1}g\|_{L^2}^{1/4}
\end{align*}
where $C>0$ is independent of $\alpha \in [1/4,1/2]$. Upon multiplying by $\tau^m M_m$ and using the definitions \eqref{eq:Xm}--\eqref{eq:Ym}, the above bound implies
\begin{align} 
\frac{|T_m| \tau^m M_m}{\|\ww \partial_x^m g\|_{L^2}} 
&\leq
\frac{C\tt^{1/2}}{(\tau(t))^{1/2}} \left( \sum_{j=0}^{[m/2]} Y_{j+1}^{1/2} Y_{j+2}^{1/2} X_{m-j}^{1/2} D_{m-j}^{1/2} + \sum_{j=[m/2]+1}^m Y_{j+1} X_{m-j}^{1/4} X_{m-j+1}^{1/4} D_{m-j}^{1/4} D_{m-j+1}^{1/4} \right) 
\notag\\
&\ + \frac{C\tt^{1/4}}{(\tau(t))^{1/2}} \left( \sum_{j=0}^{[m/2]}  Y_{j+1}^{1/2} Y_{j+2}^{1/2} X_{m-j}^{1/2} Z_{m-j}^{1/2} + \sum_{j=[m/2]+1}^m Y_{j+1} X_{m-j}^{1/4} X_{m-j+1}^{1/4} Z_{m-j}^{1/4} Z_{m-j+1}^{1/4} \right).
\label{eq:g:m:T}
\end{align}

The proof of the lemma is completed upon summing \eqref{eq:g:m:U}--\eqref{eq:g:m:T} over $m\geq 0$ and using the bound
\[
\sum_{m\geq 0} \sum_{j=0}^m a_j b_{m-j} \leq \sum_{j\geq 0} a_j   \sum_{k \geq 0} b_k 
\]
for positive sequences $\{a_j\}_{j\geq 0}$ and $\{b_j\}_{j\geq 0}$.
\end{proof}

\begin{remark}[\bf Analytic product estimates]
\label{rem:nonlinear}
We note that the proof of Lemma~\ref{lem:analytic} directly implies that the following bounds hold
\begin{align*} 
&\| \uu(g^{(1)}) \partial_x g^{(2)} \|_{X_{\tau,\alpha}} \leq \frac{C}{\tau(t)^{1/2}} \|g^{(1)}\|_{B_{\tau,\alpha}} \|g^{(2)}\|_{Y_{\tau,\alpha}} \\
&\| \vv(g^{(1)}) \partial_y g^{(2)} \|_{X_{\tau,\alpha}} \leq \frac{C}{\tau(t)^{1/2}} \|g^{(2)}\|_{B_{\tau,\alpha}} \|g^{(1)}\|_{Y_{\tau,\alpha}} \\
&\frac{1}{2\tt}\| \vv(g^{(1)}) \uu(g^{(2)}) \|_{X_{\tau,\alpha}} \leq \frac{C_0}{\tau(t)^{1/2}} \|g^{(2)}\|_{B_{\tau,\alpha}} \|g^{(1)}\|_{Y_{\tau,\alpha}} 
\end{align*}
for some universal constant $C>0$, independent of $\tau>0$ and $\alpha \in [1/4,1/2]$.
\end{remark}

\subsection{Conclusion of the a priori estimates}
\label{sec:conclusion}
At this stage we make a choice for the free parameters $\alpha,\beta$, and $\gamma$. First, we introduce 
\begin{align*} 
\delta = \delta(\eps) \in (\eps, 1/10)
\end{align*}
which is to be chosen at the end of the proof, where without loss of generality $\eps \leq 1/200$. We set
\begin{align*}
\alpha = \frac{1-\delta}{2}, \qquad  \beta = \gamma = \frac{\delta}{2}.
\end{align*}
With this choice of $\alpha,\beta,\gamma$, we sum estimate \eqref{eq:g:m:*} for $m\geq 0$, appeal to Lemmas~\ref{lem:LEFT} and~\ref{lem:analytic}, and arrive at
\begin{align*} 
&\frac{d}{dt} \|g\|_{X_{\tau,\alpha}} + \frac{5/4-\delta}{\tt} \|g\|_{X_{\tau,\alpha}} \notag\\
&\  
+\frac{\delta}{C_1 \tt^{5/4}}
\left( 2 \tt^{1/4} \|  g\|_{X_{\tau,\alpha}} +\tt^{1/4} \|g\|_{Z_{\tau,\alpha}} + \tt^{1/4} \|g\|_{\tilde Z_{\tau,\alpha}}+ \tt^{3/4} \|g\|_{D_{\tau,\alpha}} + \tt^{5/4} \|g\|_{\tilde D_{\tau,\alpha}} \right) 
\notag\\
&\ \  \leq \Bigg( \dot{\tau}(t) + \frac{C_0}{\tau(t)^{1/2}} \left( \tt^{1/4}  \|g\|_{X_{\tau,\alpha}} +   \tt^{1/4}   \| g\|_{Z_{\tau,\alpha}} +  \tt^{3/4}  \| g\|_{D_{\tau,\alpha}} \right)  \Bigg) \|g\|_{Y_{\tau,\alpha}}
\end{align*}
for some sufficiently large universal constants $C_0,C_1\geq 1$ which are independent of $\alpha$ and $\delta$. 
Upon recalling the notations \eqref{eq:B:tau} and \eqref{eq:tilde:B:tau}, we can rewrite the above in a more compact form as 
\begin{align} 
\frac{d}{dt} \|g\|_{X_{\tau,\alpha}} +\frac{5/4-\delta}{\tt} \|g\|_{X_{\tau,\alpha}} 
+\frac{\delta}{C_1 \tt^{5/4}} \left( \|g\|_{B_{\tau,\alpha}} + \|g\|_{\tilde B_{\tau,\alpha}} \right)
  \leq \left( \dot{\tau}(t) + \frac{C_0}{\tau(t)^{1/2}} \|g\|_{B_{\tau,\alpha}}  \right) \|g\|_{Y_{\tau,\alpha}}
\label{eq:X:1}
\end{align}
with $C_0, C_1 \geq 1$ are universal constants, that are in particular {\em independent} of the choice of $\delta \in (\eps,1/10)$. 

We next choose the function $\tau(t)$ such that 
\begin{align} 
\frac{d}{dt} (\tau(t))^{3/2} + 3  C_0   \|g(t)\|_{B_{\tau(t),\alpha}}  =  0.
\label{eq:tau:1}
\end{align}
The above ODE is meant to hold a.e. in time, since the time derivative of the monotone decreasing absolutely continuous function (in fact H\"older $1/2$ continuous) is only guaranteed to exist almost everywhere. With this choice of $\tau$ in \eqref{eq:tau:1}, we infer from the a priori estimate \eqref{eq:X:1} that
\begin{align*} 
\frac{d}{dt} \left( \tt^{5/4-\delta} \|g\|_{X_{\tau,\alpha}} \right)
+ \frac{\delta}{C_1 \tt^\delta} \left( \|g\|_{B_{\tau,\alpha}} + \|g\|_{\tilde B_{\tau,\alpha}} \right) 
+ \frac{C_0 \tt^{5/4-\delta}}{\tau(t)^{1/2}} \|g\|_{B_{\tau,\alpha}}   \|g\|_{Y_{\tau,\alpha}}
\leq 0
\end{align*}
which integrated on $[0,t]$ yields 
\begin{align} 
&\tt^{5/4-\delta} \|g\|_{X_{\tau(t),\alpha}} 
+ \frac{\delta}{C_1} \int_0^t \frac{1}{\ss^\delta} \left( \|g(s)\|_{B_{\tau(s),\alpha}} + \|g(s)\|_{\tilde B_{\tau(s),\alpha}} \right) ds  \notag\\
&\qquad +C_0 \int_0^t \frac{\ss^{5/4-\delta}}{\tau(s)^{1/2}} \|g(s)\|_{Y_{\tau(s),\alpha}}  \|g(s)\|_{B_{\tau(s),\alpha}} ds  \notag \\
&\quad  \leq \|g_0\|_{X_{\tau_0,\alpha}}   \leq \|g_0\|_{X_{\tau_0,1/2}} \leq \eps.
\label{eq:X:3}
\end{align}
From \eqref{eq:X:3} it immediately follows that 
\begin{align*} 
\int_0^t   \|g(s)\|_{B_{\tau(s),\alpha}} ds 
\leq   \frac{\eps C_1}{\delta}\tt^\delta 
\end{align*}
which combined with \eqref{eq:tau:1}
shows that we have the lower bound 
\begin{align} 
\tau(t)^{3/2} \geq \tau_0^{3/2} -   \frac{\eps C_2}{2 \delta} \tt^{\delta} 
\label{eq:tau:2}
\end{align}
for all $t\geq 0$, where $C_2 =  6 C_0  C_1 $ is a universal constant that is independent of $\delta$.
From estimate \eqref{eq:tau:2} we see that the radius of tangential analyticity obeys
\[
\tau(t) \geq \frac{\tau_0}{2}
\] 
on the time interval $[0,T_\eps]$, where 
\begin{align} 
\langle T_\eps \rangle^{\delta} = \frac{\delta \tau_0^{3/2}}{\eps C_2}
\label{eq:T:eps:1}
\end{align}
and we recall that $\delta = \delta(\eps) \in (\eps, 1/10)$ is yet to be chosen.

In order to see that the monotone decreasing analyticity radius is a H\"older $1/2$ continuous function of time, we may use the bound
\eqref{eq:B:tilde:B},
integrate \eqref{eq:tau:1} from $t_1$ to $t_2$, where $0\leq t_1 < t_2 \leq T_\eps$ are arbitrary, and use the estimate \eqref{eq:X:3}, to obtain
\begin{align*} 
\tau(t_1)^{3/2} - \tau(t_2)^{3/2}
&\leq 6 C_0 \int_{t_1}^{t_2} \tt^{1/8} \|g\|_{\tilde B_{\tau,\alpha}}^{1/2} \|g\|_{X_{\tau,\alpha}}^{1/2} \notag\\
&\leq 6 C_0 \sup_{t \in [0,T_\eps]} \left( \tt^{5/4-\delta} \|g\|_{X_{\tau,\alpha}}\right)^{1/2} \left( \int_{t_1}^{t_2} \tt^{-\delta} \|g\|_{\tilde B_{\tau,\alpha}} dt\right)^{1/2}  \left( \int_{t_1}^{t_2} \tt^{2\delta-1}  dt\right)^{1/2} \notag\\
&\leq \frac{6 C_0 \sqrt{C_1} \eps}{ \sqrt{\delta}} (t_2-t_1)^{1/2} 
\end{align*}
by using that $2\delta- 1 \leq 0$.

To conclude the proof, we let
\begin{align} 
\delta = {\eps} \log \frac{1}{\eps}
\label{eq:delta:def}
\end{align}
which is a permissible choice if $\eps$ is sufficiently small. In that case, from \eqref{eq:T:eps:1} we obtain
\begin{align} 
T_\eps = \left( \frac{\tau_0^{3/2} \log \frac 1 \eps}{ C_2}\right)^{\frac{1}{\eps \log \frac 1\eps}} - 1 = \exp\left( \frac{1}{\eps \log \frac 1 \eps} \log \left( \frac{\tau_0^{3/2} \log \frac 1 \eps}{C_2} \right) \right) - 1.
\label{eq:T:eps:2}  
\end{align}
It is clear from \eqref{eq:T:eps:2} that
as long  as 
$\tau_0^{3/2} \log \frac 1\eps\geq C_2  e^2 $, which is ensured by~\eqref{eq:eps:tau:0},
then we have that 
\begin{align*} 
T_\eps \geq \exp\left( \frac{1}{\eps \log \frac 1 \eps} \right)
\end{align*}
for all $0 < \eps \leq 1/200$, which concludes the proof of the a priori estimates.

\section{Uniqueness}
\label{sec:uniqueness}

Assume $g_0\in X_{2\tau_0,\alpha}$ with $\Vert g_0\Vert_{X_{2\tau_0},\alpha}\le \eps$. 
Let $g^{(1)}$ and $g^{(2)}$ be two solutions to the system \eqref{eq:g:1}--\eqref{eq:g:3} evolving from $g_0$, with tangential radii of analyticity $\tau^{(1)}$ and $\tau^{(2)}$ respectively, which obey the bounds in Theorem~\ref{thm:rigorous}. 
We fix $\delta$ as given by \eqref{eq:delta:def}.

Also, define $\tau(t)$ by
\begin{align}
\dot{\tau}(t)+\frac{2C_0}{\tau(t)^{1/2}}\Vert g^{(1)}(t)\Vert_{B_{\tau^{(1)}(t)}}=0, \qquad \tau(0)=\frac{\tau_0}{4}.
\label{eq:diff:tau:0}
\end{align}
In view of the estimate \eqref{eq:g:bnd:2} for $g^{(1)}$ and the lower bounds \eqref{eq:g:bnd:0} for $\tau^{(1)}$ and $\tau^{(2)}$, we have that 
\begin{align}
\frac{\tau_0}{8} \leq \tau(t) \leq \frac{\tau_0}{4} \leq \frac{\min\{\tau^{(1)},\tau^{(2)}\}}{2}
\label{eq:diff:tau:1}
\end{align}
for all $t\in [0,T_\eps]$.

We consider the difference of solutions $\bar g=g^{(1)}-g^{(2)}$ which obeys 
\begin{align}
&\partial_t \bar g-\partial_y^2 \bar g+\kappa\phi\partial_x \bar g+\frac{1}{\tt}\bar g\notag\\
&\qquad=-(u^{(1)}\partial_x \bar g+\bar u\partial_x g^{(2)}) - (\bar v\partial_y g^{(1)}+v^{(2)}\partial_y \bar g)
+\frac{1}{2\tt}(\bar v u^{(1)}+v^{(2)} \bar u)
\label{eq:diff:1}
\end{align}
and has initial datum $\bar g_0=0$.
Here we also denote $\bar u=u^{(1)}-u^{(2)}=\uu(\bar g)$ and $\bar v=v^{(1)}-v^{(2)}=\vv(\bar g)$. 

Using estimates for the nonlinear terms as in Remark~\ref{rem:nonlinear}, similarly to \eqref{eq:X:1} we arrive at
\begin{align}
&\frac{d}{dt}\Vert \bar g(t)\Vert_{X_{\tau(t)}}+\frac{5/4-\delta}{\tt}\Vert \bar g(t)\Vert_{X_{\tau(t)}}
+\frac{\delta}{C_1\tt^{5/4}}\Vert \bar g(t)\Vert_{B_{\tau(t)}}\notag\\
&\quad\le 
\left(\dot{\tau}(t)+ \frac{2 C_0}{\tau(t)^{1/2}}\Vert g^{(1)}(t)\Vert_{B_{\tau(t)}} \right)\Vert \bar g(t)\Vert_{Y_{\tau(t)}} +
\frac{2 C_0}{\tau(t)^{1/2}}\Vert g^{(2)}(t)\Vert_{Y_{\tau(t)}}\Vert \bar g(t)\Vert_{B_\tau(t)}
\label{eq:diff:2}
\end{align} 
with $C_0,C_1\geq 1$ being universal constants. Since $\tau(t) \leq \tau^{(1)}(t)$ and the $X_{\tau,\alpha}$ norm is increasing in $\tau$, we obtain from \eqref{eq:diff:tau:0} that 
\[
\dot{\tau}(t)+ \frac{2 C_0}{\tau(t)^{1/2}}\Vert g^{(1)}(t)\Vert_{B_{\tau(t)}} \leq 0.
\]
On the other hand, using \eqref{eq:Y:X:bound}, \eqref{eq:diff:tau:1}, and \eqref{eq:g:bnd:1} we may bound
\[
\Vert g^{(2)}(t)\Vert_{Y_{\tau(t)}} \leq \frac{1}{\tau(t)} \Vert g^{(2)}(t)\Vert_{X_{2\tau(t)}} \leq \frac{1}{\tau(t)} \Vert g^{(2)}(t)\Vert_{X_{\tau^{(1)}(t)}}  \leq \frac{\eps}{\tt^{5/4-\delta} \tau(t)}.
\]
Combining the above two estimates with \eqref{eq:diff:2} we arrive at
\begin{align}
\frac{d}{dt}\Vert \bar g(t)\Vert_{X_{\tau(t)}}+\frac{5/4-\delta}{\tt}\Vert \bar g(t)\Vert_{X_{\tau(t)}}
+\frac{\delta}{C_1\tt^{5/4}}\Vert \bar g(t)\Vert_{B_{\tau(t)}}
\leq \frac{2 \eps C_0 \tt^{\delta}}{\tt^{5/4} \tau(t)^{3/2}} \| \bar g(t)\|_{B_\tau(t)}
\label{eq:diff:3}.
\end{align} 
To conclude we note that by the definition of $T_\eps$ in \eqref{eq:T:eps:1} we have that
\begin{align}
 \frac{\delta}{C_1} =\frac{\eps \log \frac 1 \eps}{C_1}  \geq \frac{32 \eps C_0 \tt^{\delta}}{\tau_0^{3/2}} \geq \frac{2 \eps C_0 \tt^{\delta}}{\tau(t)^{3/2}} 
 \label{eq:diff:4}
\end{align}
holds. 
From \eqref{eq:diff:3} and \eqref{eq:diff:4} we obtain
\begin{align*}
\frac{d}{dt}\Vert \bar g(t)\Vert_{X_{\tau(t)}}+\frac{5/4-\delta}{\tt}\Vert \bar g(t)\Vert_{X_{\tau(t)}}
\leq 0
\end{align*} 
which concludes the proof of uniqueness since $\bar g_0 = 0$.

\section{Existence}
\label{sec:existence}

Throughout this section we fix $\alpha = 1/2 - \delta$, where $\delta = \eps \log \frac{1}{\eps}$. We assume the initial datum $g_0$ obeys $\|g_0\|_{X_{2\tau_0,1/2}} \leq \eps$, where the pair $(\tau_0,\eps)$ obeys \eqref{eq:eps:tau:0}. 

We first prove the existence of solutions $g^{(\nu)}$ to a parabolic approximation of the Prandtl equations, with the term $-\nu \partial_x^2 g$ present on the left side of \eqref{eq:g:1}. These solutions are shown to obey uniform in $\nu$ bounds in $L^\infty_t X_{\tau^{(\nu)}(t),\alpha} \cap L^1_t B_{\tau^{(\nu)}(t),\alpha}$ for a sequence of tangential analyticity radii $\tau^{(\nu)}$. These radii obey $\tau^{(\nu)}\geq \tau_0/2$ for all $t \in [0,T_\eps]$ and are moreover uniformly equicontinuous on this time interval, where $T_\eps$ is given by \eqref{eq:T:eps:1}, i.e.
\begin{align}
\langle T_\eps \rangle^\delta = \frac{\tau_0^{3/2} \log \frac 1 \eps}{K_*} 
\label{eq:T:eps:*}
\end{align}
 for a sufficiently large universal constant $K_*$. Moreover, $g^{(\nu)}$ and $\tau^{(\nu)}$ are shown to obey \eqref{eq:tau:1}.
 
With these uniform in $\nu$ bounds we then show that the $\tau^{(\nu)}$ converge along a subsequence to an analyticity radius $\tau(t) \geq \tau_0/2$ on $[0,T_\eps]$, and along this subsequence, the $g^{(\nu)}$ are shown to be a Cauchy sequence in the topology induced by $L^\infty_t X_{\tau_0,\alpha} \cap L^1_t B_{\tau_0,\alpha}$. By the completeness of $L^\infty( L^2(\theta_\alpha(t,y) dy dx) dt)$ the existence of solutions to Prandtl in the sense of Definition~\ref{def:sol} is then completed.

\subsection{A dissipative approximation}
For $\nu>0$ we consider
the nonlinear parabolic equation
\begin{align} 
&\partial_t \gnu - \partial_y^2 \gnu - \nu \partial_x^2 \gnu + (\unu + \kappa \phi) \partial_x \gnu  + \vnu \partial_y \gnu + \frac{1}{\tt} \gnu - \frac{1}{2\tt} \vnu \unu  = 0   \label{eq:gnu:1}\\
&\partial_y \gnu|_{y=0}= \gnu|_{y=\infty}  \label{eq:gnu:2}\\
&\unu(y) = \uu(\gnu) := \theta_{-1}(y) \int_0^y \gnu(\bar y) \theta_1(\bar y) d\bar y \label{eq:gnu:3}\\
&\vnu(y) = \vv(\gnu) := - \int_0^y \partial_x \unu(\bar y) d \bar y  \label{eq:gnu:4}.
\end{align}
Our goal is to construct solutions $g^{(\nu)}$
with corresponding tangential analyticity radii $\tau^{(\nu)}$, so that uniformly in $\nu >0$ we have the estimate
\begin{align}
&\sup_{t \in [0,T_\eps]} \left(\tt^{5/4-\delta} \|\gnu\|_{X_{\tau^{(\nu)}(t),\alpha}} \right)
+ \frac{\delta}{K_*} \int_0^{T_\eps} \frac{1}{\ss^\delta}\| \gnu(s)\|_{B_{\tau^{(\nu)}(s),\alpha}}  ds  
\notag \\
& \qquad+ K_* \int_0^{T_\eps} \frac{\ss^{5/4-\delta}}{\tau(s)^{1/2}} \|\gnu(s)\|_{Y_{\tau^{(\nu)}(s),\alpha}} \| \gnu(s)\|_{B_{\tau^{(\nu)}(s),\alpha}} ds   \leq 4 \eps,
\label{eq:gnu:bnd} 
\end{align}
where $K_*>0$ is a sufficiently large universal constant, and the radii $\tau^{(\nu)}(t)$ obey the ODE
\begin{align}
\frac{d}{dt} \tau^{(\nu)}(t) + \frac{2 K_*}{\tau^{(\nu)}(t)^{1/2}} \| \gnu(t)\|_{B_{\tau^{(\nu)}(t),\alpha}} = 0, \qquad 
\tau^{(\nu)}(0)=\tau_0.
\label{eq:tau:nu}
\end{align}
For $\nu>0$, estimate \eqref{eq:gnu:bnd} and the ODE \eqref{eq:tau:nu}, correspond to \eqref{eq:X:3} respectively \eqref{eq:tau:1} for the limiting Prandtl system $\nu =0$. 
Although the system \eqref{eq:gnu:1}--\eqref{eq:gnu:4} is parabolic, we detail the construction of $g^{(\nu)}$ and $\tau^{(\nu)}$ since the first order ODE \eqref{eq:tau:nu} has a nonlinear term which convergences only once the radius $\tau^{(\nu)}$ has been constructed already to satisfy this equation. 
The method of constructing $g^{(\nu)}$ and $\tau^{(\nu)}$ draws from ideas employed~\cite{KukavicaTemamVicolZiane11,IgnatovaKukavicaZiane12} for the hydrostatic Euler equations. 

At this stage it is convenient to introduce some notation. Let $N \geq 1$.
Similarly to \eqref{eq:Xm}--\eqref{eq:tilde:B:tau}, for $h \colon \HH \to \RR$ and $\tau >0$ define the weighted Sobolev norms
\begin{alignat}{2}
&\|h\|_{X^{N}_\tau} = \sum_{m=0}^N X_{m}(h,\tau), 
\qquad 
& &
\label{eq:XmN} \\
&\|h\|_{D^{N}_\tau}= \sum_{m=0}^N D_{m}(h,\tau), 
\qquad
& & \|h\|_{\tilde D^{N}_\tau} = \sum_{m=0}^N \frac{D_{m}(h,\tau)^2}{X_{m}(h,\tau)} = \sum_{m=0}^N \tilde D_m(h,\tau),
\notag \\
&\|h\|_{Z^{N}_\tau} = \sum_{m=0}^N Z_{m}(h,\tau), 
\qquad
& & \|h\|_{\tilde Z^{N}_\tau} = \sum_{m=0}^N \frac{Z_{m}(h,\tau)^2}{X_{m}(h,\tau)} = \sum_{m=0}^N \tilde Z_m(h,\tau),
\notag  \\
&\|h\|_{Y^{N}_\tau} = \sum_{m=1}^N Y_{m}(h,\tau),
\qquad 
& & \|h\|_{\tilde Y^{N}_\tau} = \sum_{m=1}^N \frac{Y_{m}(h,\tau)^2}{X_{m-1}(h,\tau)},
\notag  \\
& \|h\|_{B^{N}_\tau} =\sum_{m=0}^N B_{m}(h,\tau), 
\qquad
& &  \|h\|_{\tilde B^{N}_\tau} =\sum_{m=0}^N \tilde B_{m}(h,\tau). \label{eq:tildeBmN}
\end{alignat}
We will use frequently that the bound
\begin{align*} 
\|h\|_{B^{N}_\tau}^2 \leq 3 \tt^{1/4} \|h\|_{X^{N}_\tau}  \|h\|_{\tilde B^{N}_\tau} 
\end{align*}
holds independently of $N\geq 1$ and $\tau>0$.

\subsection{A two-step Picard iteration for the dissipative system}
We define
\begin{align*}
S^{(\nu)}(t) h_0 = h^{(\nu)}(t)
\end{align*}
to be the solution to the initial value problem to the linear part of \eqref{eq:gnu:1}--\eqref{eq:gnu:4}, namely
\begin{align}
&\partial_t  h^{(\nu)} - \partial_y^2 h^{(\nu)} - \nu \partial_x^2 h^{(\nu)} + \kappa \phi \partial_x h^{(\nu)} + \frac{1}{\tt} h^{(\nu)} = 0   \label{eq:hnu:1}\\
&\partial_y h^{(\nu)}|_{y=0} = 0 = h^{(\nu)}|_{y=\infty}  
\label{eq:hnu:2}\\
&h^{(\nu)}|_{t=0} = h_0 
\label{eq:hnu:3}.
\end{align}
Solving \eqref{eq:hnu:1}--\eqref{eq:hnu:3} on $\HH$ with the Neumann boundary condition \eqref{eq:hnu:2} at $y=0$ may be done using an even extension across $y=0$ and solving the problem \eqref{eq:hnu:1} on $\RR^2$ with vanishing boundary conditions as $|y|\to \infty$. As such, an explicit solution formula for $S^{(\nu)}(t)$ may be obtained, though it will not be essentially used here. We note that if $h_0$ obeys the boundary condition \eqref{eq:hnu:2}, the solutions  $S^{(\nu)}(t) h_0$ automatically lie in $\HCal_{2,1,\beta}$ for any $\beta < 1$ (cf.~Definition~\ref{def:sol}).

Next, we set up a two-step Picard iteration scheme. For $n=0,1$ we let
\begin{align*}
g^{(0,\nu)}(t) = g^{(1,\nu)}(t) = S^{(\nu)}(t) g_0 
\end{align*}
while for $n\geq 2$ we define $g^{(n,\nu)}$ to be the mild solution (obtained by the Duhamel formula for the semigroup $S^{(\nu)}$) of the linear  initial value problem
\begin{align} 
&\partial_t g^{(n,\nu)}- \partial_y^2 g^{(n,\nu)} - \nu \partial_x^2 g^{(n,\nu)} + \kappa \phi \partial_x g^{(n,\nu)} + \frac{1}{\tt} g^{(n,\nu)} \notag\\
&\qquad = - \uu(g^{(n-2,\nu)}) \partial_x g^{(n-1,\nu)} - \vv(g^{(n-1,\nu)}) \partial_y g^{(n-2,\nu)}  + \frac{1}{2\tt} \vv(g^{(n-1,\nu)}) \uu(g^{(n-2,\nu)}) \label{eq:Picard:1}\\
&\partial_y g^{(n,\nu)}|_{y=0} = 0 =g^{(n,\nu)}|_{y=\infty}  \label{eq:Picard:2}\\
&g^{(n,\nu)}|_{t=0} = g_0 \label{eq:Picard:3}.
\end{align}
The pairing of $g^{(n-1,\nu)}$ and $g^{(n-2,\nu)}$ in \eqref{eq:Picard:1} is motivated by the bounds guaranteed by Remark~\ref{rem:nonlinear}.

\subsection{Sobolev bounds and convergence of the Picard iteration}
Let $N$ be an integer such that 
$
N \geq \frac{1}{\nu}.
$
For the remainder of this subsection we fix this value of $N$ and we shall ignore the $\nu$ and $N$ indices for $g$ and $\tau$. 
We claim that there exists $T_{\eps,N}>0$, to be chosen later, and a sequence of absolutely continuous monotone decreasing functions
\begin{align}
\tau^{(n)} = \tau^{(n)}_N \colon [0,T_{\eps,N}] \to \left[\frac{5 \tau_0}{4} , \frac{ 7\tau_0}{4}\right]
\label{eq:tau:n:bnd}
\end{align}
with $\tau^{(n)}(0) = 7 \tau_0/4$
such that the bound
\begin{align}
&\sup_{[0,T_{\eps,N}]} \left( \tt^{5/4-\delta} \| g^{(n)}(t)\|_{X_{\tau^{(n)}(t)}^{N}} \right)
+ \frac{\delta}{K} \int_0^{T_{\eps,N}} \frac{1}{\ss^\delta} \left( \| g^{(n)}(s)\|_{B_{\tau^{(n)}(s)}^{N}} +   \| g^{(n)}(s)\|_{\tilde B_{\tau^{(n)}(s)}^{N}} \right) ds  \notag \\
&\  + \frac{8^N K}{\tau_0^{1/2}} \int_0^{T_{\eps,N}}  \ss^{5/4-\delta} \left( \|g^{(n-1)}(s)\|_{B_{\tau^{(n)}(s)}^{N}} + \|g^{(n-1)}(s)\|_{\tilde B_{\tau^{(n)}(s)}^{N}} \right) \|g^{(n)}\|_{Y_{\tau^{(n)}(s)}^N} ds \notag\\
&\ +\frac{\nu}{4} \int_0^{T_{\eps,N}} \ss^{5/4-\delta} \|g^{(n)}(s)\|_{\tilde Y^{N+1}_{\tau^{(n)}(s)}} ds \notag\\
&\leq 2 \eps \label{eq:cc} 
\end{align}
holds for all $n \geq 1$, and some universal constant $K\geq 1$.

We prove \eqref{eq:cc} inductively on $n$. For $n=1$ this bound follows immediately from the assumption $\|g_0\|_{X_{2\tau_0,1/2}} \leq \eps$, and the dissipativity of $S^{(\nu)}$.
In order to prove the induction step we proceed as follows.
Since $M_{m-1}/2 \leq m M_m \leq 2 M_{m-1}$ for all $m\geq 1$, and there are no boundary terms when integrating by parts in $x$, for all $m\geq 0$ one may use Remark~\ref{rem:nonlinear} to derive an estimate similar to \eqref{eq:X:1} for the system \eqref{eq:Picard:1}--\eqref{eq:Picard:3} which is
\begin{align}
&\frac{d}{dt} \|g^{(n)}\|_{X^{N}_{\tau^{(n)}}} + \frac{5/4-\delta}{\tt} \|g^{(n)}\|_{X^{N}_{\tau^{(n)}}} 
+ \frac{\delta}{K \tt^{5/4}} \left( \|g^{(n)}\|_{B^{N}_{\tau^{(n)}}} + \|g^{(n)}\|_{\tilde B^{N}_{\tau^{(n)}}} \right) \notag\\
&\quad + \frac{8^N K}{\tau_0^{1/2}} \left( \|g^{(n-1)}\|_{B^{N}_{\tau^{(n)}}} +  \|g^{(n-1)}\|_{\tilde B^{N}_{\tau^{(n)}}} \right) \|g^{(n)}\|_{Y_{\tau^{(n)}}^{N}}  
+\frac{\nu}{4} \|g^{(n)}\|_{\tilde Y^{N+1}_{\tau^{(n)}}} 
\notag\\
& \leq \left(\dot \tau^{(n)}  + \frac{8^N K}{\tau_0^{1/2}}  \left( \|g^{(n-1)}\|_{B^{N}_{\tau^{(n)}}} + \|g^{(n-1)}\|_{\tilde B^{N}_{\tau^{(n)}}} \right) \right) \|g^{(n)}\|_{Y_{\tau^{(n)}}^{N}} 
\notag\\
&\quad +  \frac{K}{(\tau^{(n)})^{1/2}} \|g^{(n-2)}\|_{B^{N}_{\tau^{(n)}}} \|g^{(n-1)}\|_{Y^{N}_{\tau^{(n)}}}  + \frac{K}{(\tau^{(n)})^{1/2}} \|g^{(n-2)}\|_{B^{N}_{\tau^{(n)}}} Y_{N+1}(g^{(n-1)},\tau^{(n)}) 
\notag\\
& \leq \left(\dot \tau^{(n)}  + \frac{12^N K}{\tau_0^{1/2}}  \left( \|g^{(n-1)}\|_{B^{N}_{\tau^{(n-1)}}} + \|g^{(n-1)}\|_{\tilde B^{N}_{\tau^{(n-1)}}} 
\right) \right) \|g^{(n)}\|_{Y_{\tau^{(n)}}^{N}}  
\notag\\
&\quad +  \frac{2^{N} K}{\tau_0^{1/2}} \|g^{(n-2)}\|_{B^{N}_{\tau^{(n-1)}}} \|g^{(n-1)}\|_{Y^{N}_{\tau^{(n-1)}}} 
 + \frac{2^{N} K}{\tau_0^{1/2}} \|g^{(n-2)}\|_{B^{N}_{\tau^{(n-1)}}} Y_{N+1}(g^{(n-1)},\tau^{(n-1)}) 
\label{eq:constr:1}
\end{align}
for all $n\geq 2$, 
where $K$ is a sufficiently large universal constant (in particular $\delta,N,n,\tau$-independent). In the second inequality in \eqref{eq:constr:1} we have used several times that cf.~\eqref{eq:tau:n:bnd} 
we have
\[
\max_{0 \leq |j| \leq N} \left( \frac{\tau^{(n)}}{\tau^{(n-1)}} \right)^{j} \leq \max_{0 \leq |j| \leq N} \left(\frac{7}{5}\right)^{j} \leq 2^{N/2}.
\]
The main difficulty lies in obtaining a $\nu-$independent bound for the last term on the right side of \eqref{eq:constr:1}. First we notice that since 
\[
\frac{X_m(h,\tau)}{\tau} = \frac{Y_m(h,\tau)}{m} 
\]
and $N \nu \geq 1$ we may estimate
\begin{align}
&\frac{2^{N+2} K}{\tau_0^{1/2}} \|g^{(n-2)}\|_{B^{N}_{\tau^{(n-1)}}} Y_{N+1}(g^{(n-1)},\tau^{(n-1)}) \notag\\
&\leq \frac{\nu}{8} \frac{(Y_{N+1}(g^{(n-1)},\tau^{(n-1)})^2}{X_N(g^{(n-1)},\tau^{(n-1)})} 
+ \frac{4^{N+3} K^2}{\nu \tau^{(n-1)}} \|g^{(n-2)}\|_{B^{N}_{\tau^{(n-1)}}}^2 X_N(g^{(n-1)},\tau^{(n-1)}) \notag\\
&\leq \frac{\nu}{8} \|g^{(n-1)}\|_{\tilde Y^{N+1}_{\tau^{(n-1)}}} + \frac{4^{N+4} K^2}{\nu N} \|g^{(n-2)}\|_{B^{N}_{\tau^{(n-1)}}}^2 Y_N(g^{(n-1)},\tau^{(n-1)})  \notag\\
&\leq \frac{\nu}{8} \|g^{(n-1)}\|_{\tilde Y^{N+1}_{\tau^{(n-1)}}} + \frac{4^{N+5} K^2 \tt^{1/4}}{\nu N} \|g^{(n-2)}\|_{X^N_{\tau^{(n-1)}}} \|g^{(n-2)}\|_{\tilde B^{N}_{\tau^{(n-1)}}} \| g^{(n-1)}\|_{Y^N_{\tau^{(n-1)}}} \notag\\
&\leq \frac{\nu}{8} \|g^{(n-1)}\|_{\tilde Y^{N+1}_{\tau^{(n-1)}}} + 8^{N+4} K^2 \tt^{1/4} \|g^{(n-2)}\|_{X^N_{\tau^{(n-2)}}} \|g^{(n-2)}\|_{\tilde B^{N}_{\tau^{(n-1)}}} \| g^{(n-1)}\|_{Y^N_{\tau^{(n-1)}}} 
\label{eq:constr:2}.
\end{align}
At this stage, for $n \geq 1$ we chose $\tau^{(n)}$ to solve the first order ODE
\begin{align}
\dot \tau^{(n)}  + \frac{12^N K}{\tau_0^{1/2}}  \left( \|g^{(n-1)}\|_{B^{N}_{\tau^{(n-1)}}} + \|g^{(n-1)}\|_{\tilde B^{N}_{\tau^{(n-1)}}} \right) = 0, \qquad \tau^{(n)}(0) = \frac{7 \tau_0}{4}.
\label{eq:tau:n:ODE}
\end{align}
The key point here is that by the induction step, the functions $g^{(n-1)}$ and $\tau^{(n-1)}$ are known, and due to the estimates \eqref{eq:cc} we have that 
\begin{align}
\int_0^t \|g^{(n-1)}(s)\|_{B^{N}_{\tau^{(n-1)}(s)}} + \|g^{(n-1)}(s)\|_{\tilde B^{N}_{\tau^{(n-1)}(s)}} ds \leq \frac{2\eps K \tt^\delta}{\delta} < \infty
\label{eq:tau:n:derivative}
\end{align}
for all $t \in [0,T_{\eps,N}]$.
Thus the existence of an absolutely continuous solution $\tau^{(n)}$ to \eqref{eq:tau:n:ODE} is immediate. For $n=0$ we may simply let $\tau^{(0)}(t) = 7 \tau_0/4$. Moreover, from \eqref{eq:tau:n:derivative} we have that \eqref{eq:tau:n:bnd} holds at least on $[0,T_{\eps,N}]$, with $T_{\eps,N}$ defined by
\begin{align}
\tau_0^{3/2} = \frac{12^{N+1} K^2 }{\log \frac{1}{\eps}} \langle T_{\eps,N} \rangle^{\delta}.
\label{eq:tau:eps:n}
\end{align}
In fact we will a-posteriori show that the time interval can be chosen independently of $N \geq 1/\nu$, as the factor $12^N$ is superfluous. For the moment however, the bound \eqref{eq:tau:eps:n} is good enough since it is independent of $n \geq 0$ (recall that for now $N$ is fixed).

We now combine the bounds \eqref{eq:constr:1} and \eqref{eq:constr:2} with the choice for $\tau^{(n)}$ made in \eqref{eq:tau:n:ODE}, integrate on $[0,t]$, and use the induction assumption (via the bounds \eqref{eq:cc}) to obtain that
\begin{align}
&\tt^{5/4-\delta} \|g^{(n)}(t)\|_{X_{\tau^{(n)}(t)}^{N}} 
+ \frac{\delta}{K} \int_0^t \frac{1}{\ss^\delta}  
\left( \| g^{(n)}(s)\|_{B_{\tau^{(n)}(s)}^{N}}+ \| g^{(n)}(s)\|_{\tilde B_{\tau^{(n)}(s)}^{N}} \right) ds  
\notag\\
&\quad + \frac{8^N K}{\tau_0^{1/2}} \int_0^t \ss^{5/4-\delta} \left( \|g^{(n-1)}(s)\|_{B^{N}_{\tau^{(n)}(s)}} +  \|g^{(n-1)}(s)\|_{\tilde B^{N}_{\tau^{(n)}(s)}} \right) \|g^{(n)}(s)\|_{Y_{\tau^{(n)}(s)}^{N}} ds 
\notag\\
&\quad + \frac{\nu}{4} \int_0^t \ss^{5/4-\delta}  \| g^{(n)}(s)\|_{\tilde Y_{\tau^{(n)}(s)}^{N+1}} ds  
\notag\\
&\leq \|g_0\|_{X_{2\tau_0}^N} + \frac{\nu}{8} \int_0^t \ss^{5/4-\delta}  \| g^{(n-1)}(s)\|_{\tilde Y_{\tau^{(n-1)}(s)}^{N+1}} ds 
\notag\\
&\quad +  8^{N+4} K^2  \left( \sup_{[0,t]} \ss^{5/4-\delta} \|g^{(n-2)(s)}\|_{X^{N}_{\tau^{(n-2)}(s)}}\right)  \left(\sup_{[0,t]} \ss^{2\delta-1}\right)
\notag\\
&\qquad \qquad \qquad \times  \int_0^t  \ss^{5/4-\delta}  \|g^{(n-2)}(s)\|_{\tilde B^{N}_{\tau^{(n-1)}(s)}} \|g^{(n-1)}(s)\|_{Y^{N}_{\tau^{(n-1)}(s)}}ds  
\notag\\
&\quad + \frac{2^{N+2} K}{\tau_0^{1/2}} \int_0^t  \ss^{5/4-\delta} \|g^{(n-2)}(s)\|_{B^{N}_{\tau^{(n-1)}(s)}} \|g^{(n-1)}(s)\|_{Y^{N}_{\tau^{(n-1)}(s)}} ds 
\notag\\
&\leq \eps + \frac{\nu}{8} \int_0^t \ss^{5/4-\delta}  \| g^{(n-1)}(s)\|_{\tilde Y_{\tau^{(n-1)}(s)}^{N+1}} ds\notag\\
&\quad + \frac{8^{N} K}{2 \tau_0^{1/2}}
\int_0^t  {\ss^{5/4-\delta}} 
\left( \|g^{(n-2)}(s)\|_{B^{N}_{\tau^{(n-1)}(s)}} + \|g^{(n-2)}(s)\|_{\tilde B^{N}_{\tau^{(n-1)}(s)}}  \right)  \|g^{(n-1)}(s)\|_{Y^{N}_{\tau^{(n-1)}(s)}} ds\notag\\
&\leq 2 \eps
\label{eq:constr:3}
\end{align}
holds for all $t\in [0,T_{\eps,N}]$.
In the second to last inequality above we have used that 
\[
\max \{8^4 K\eps  \tau_0^{1/2}, 4^{-N+2} \} \leq 1,
\]
which holds if $N \geq 2$ and $\tau_0^{1/2} \eps$ is less than a small universal constant (which was assumed in \eqref{eq:eps:tau:0}). This concludes the proof of the $n$-independent bounds \eqref{eq:tau:n:bnd}--\eqref{eq:cc} for the $g^{(n)} = g^{(n,\nu)}$. 

In order to show that the Picard approximation converges, we next show that the difference 
\begin{align*}
\bar g^{(n)} = g^{(n)} - g^{(n-1)}
\end{align*}
contracts exponentially in a suitable weighted Sobolev space, of order $N-1$ in $x$. For this purpose, for $n\geq 1$ define 
the decreasing function $\bar \tau^{(n)}(t)$ by 
\begin{align}
&\frac{d}{dt} {\bar \tau^{(n)}}(t) + \frac{8^N K}{\tau_0^{1/2}} \left( \|g^{(n-1)}(t)\|_{B^{N-1}_{5 \tau_0/4}} + \|g^{(n-1)}(t)\|_{\tilde B^{N-1}_{5 \tau_0/4}}  \right) = 0, \qquad \bar \tau^{(n)}(0) = \frac{7 \tau_0}{4},
\label{eq:bar:tau:n}
\end{align}
which by the uniform in $n$ estimate \eqref{eq:cc} obeys 
\[
\bar \tau^{(n)}(t) \geq \frac{5 \tau_0}{4} \qquad \mbox{for} \qquad t\in [0,T_{\eps,N}].
\]
We measure the difference $\bar g^{(n)}$ by
\begin{align*}
A_n :=&\sup_{[0,T_{\eps,N}]} \left( \tt^{5/4-\delta} \| \bar g^{(n)}(t)\|_{X_{\bar \tau^{(n)}(t)}^{N-1}} \right) + \frac{\nu}{4} \int_0^{T_{\eps,N}} \ss^{5/4-\delta} \|\bar  g^{(n)}(s)\|_{\tilde Y^{N}_{\bar \tau^{(n)}(s)}} ds \notag\\
&+ \frac{\delta}{K} \int_0^{T_{\eps,N}} \frac{1}{\ss^\delta} \left( \|\bar g^{(n)}(s)\|_{B_{\bar \tau^{(n)}(s)}^{N-1}} +   \|\bar  g^{(n)}(s)\|_{\tilde B_{\bar \tau^{(n)}(s)}^{N-1}} \right) ds  \notag \\
&\  + \frac{8^N K}{\tau_0^{1/2}} \int_0^{T_{\eps,N}}  \ss^{5/4-\delta} \left( \|g^{(n-1)}(s)\|_{B_{5\tau_0/4}^{N-1}} + \|g^{(n-1)}(s)\|_{\tilde B_{5\tau_0/4}^{N-1}} \right) \|\bar g^{(n)}\|_{Y_{\bar \tau^{(n)}(s)}^{N-1}} ds.
\end{align*}
We claim that the sequence $A_n$ contracts, and prove that 
\begin{align}
A_n \leq \frac{A_{n-1} + A_{n-2}}{4} \label{eq:An:contract}
\end{align}
for all $n \geq 2$. In order to establish \eqref{eq:An:contract} we consider the equation obeyed by $\bar g^{(n)}$ 
\begin{align*}
&\partial_t \bar g^{(n)} - \partial_y^2 \bar g^{(n)} - \nu \partial_x^2 \bar g^{(n)} + \kappa \phi \partial_x \bar g^{(n)} + \frac{1}{\tt} \bar g^{(n)} 
\notag\\
&\qquad = - \uu(g^{(n-2)}) \partial_x \bar g^{(n-1)} - \uu(\bar g^{(n-2)}) \partial_x  g^{(n-2)}
- \vv(\bar g^{(n-1)}) \partial_y g^{(n-2)} - \vv(g^{(n-2)}) \partial_y  \bar g^{(n-2)}
\notag\\
&\qquad \quad +\frac{1}{2\tt} \vv(\bar g^{(n-1)}) \uu(g^{(n-2)}) +\frac{1}{2\tt} \vv(g^{(n-2)}) \uu(\bar g^{(n-2)}).
\end{align*}
Using estimates that are similar to those in Remark~\ref{rem:nonlinear}, the bounds \eqref{eq:constr:1}--\eqref{eq:constr:3}, and using the choice of $\bar \tau^{(n)}$ in \eqref{eq:bar:tau:n} it then follows that 
\begin{align}
&\frac{d}{dt} \| \bar g^{(n)}\|_{X^{N-1}_{\bar \tau^{(n)}}} + \frac{5/4-\delta}{\tt} \| \bar g^{(n)}\|_{X^{N-1}_{\bar \tau^{(n)}}} + \frac{\delta}{K \tt^{5/4}} \left( \| \bar g^{(n)}\|_{B^{N-1}_{\bar \tau^{(n)}}} + \| \bar g^{(n)}\|_{\tilde B^{N-1}_{\bar \tau^{(n)}}} \right) + \frac{\nu}{4} \| \bar g^{(n)}\|_{\tilde Y^{N}_{\bar \tau^{(n)}}}
\notag\\
&\quad + \frac{8^N K}{\tau_0^{1/2}} \left( \| g^{(n-1)}\|_{B^{N-1}_{5\tau_0/4}} + \|  g^{(n-1)}\|_{\tilde B^{N-1}_{5\tau_0/4}} \right) \| \bar g^{(n)}\|_{Y^{N-1}_{\bar \tau^{(n)}}}
\notag\\
&\leq \left(\frac{d}{dt} \bar \tau^{(n)} + \frac{8^N K}{\tau_0^{1/2}}\left(  \| g^{(n-1)}\|_{B^{N-1}_{5\tau_0/4}} + \|  g^{(n-1)}\|_{\tilde B^{N-1}_{5\tau_0/4}} \right) \right)  \| \bar g^{(n)}\|_{Y^{N-1}_{\bar \tau^{(n)}}}
\notag\\
&\quad + \frac{2^N K}{\tau_0^{1/2}} \| \bar g^{(n-2)}\|_{B^{N-1}_{\bar \tau^{(n-2)}}} \|g^{(n-2)}\|_{Y^{N}_{\tau^{(n-2)}}}
+ \frac{\nu}{16}  \|\bar g^{(n-1)}\|_{\tilde Y^{N}_{{\bar \tau^{(n-1)}}}}  
\notag\\
&\quad + \left(  \frac{4^N K}{\tau_0^{1/2}} \| g^{(n-2)} \|_{B^{N-1}_{5\tau_0/4}} + 8^{N} K^2 \tt^{1/4} \|g^{(n-2)}\|_{X^{N-1}_{5\tau_0/4}} \|g^{(n-2)}\|_{\tilde B^{N-1}_{5 \tau_0/4}}  \right)  \| \bar g^{(n-1)}\|_{Y^{N-1}_{\bar \tau^{(n-1)}}}  
\notag\\
&\leq  \frac{N 2^{N+1} K}{\tau_0^{3/2}} \| \bar g^{(n-2)}\|_{B^{N-1}_{\bar \tau^{(n-2)}}} \|g^{(n-2)}\|_{X^{N}_{\tau^{(n-2)}}}
+ \frac{\nu}{16}  \|\bar g^{(n-1)}\|_{\tilde Y^{N}_{{\bar \tau^{(n-1)}}}}  
\notag\\
&\quad + \frac 14 \left(  \frac{8^N K}{\tau_0^{1/2}}  \left( \| g^{(n-2)} \|_{B^{N-1}_{5\tau_0/4}} + \|g^{(n-2)}\|_{\tilde B^{N-1}_{5 \tau_0/4}}\right)  \right)  \| \bar g^{(n-1)}\|_{Y^{N-1}_{\bar \tau^{(n-1)}}}  
\label{eq:constr:4}
\end{align}
for $t \in [0,T_{\eps,N}]$. The proof of \eqref{eq:An:contract}  now follows from \eqref{eq:constr:4} upon integrating in time, recalling that $\delta = \eps \log \frac{1}{\eps}$,  that $(\eps, \tau_0)$ obey \eqref{eq:eps:tau:0}, and that the bound
\[
\frac{N 2^{N+1} K}{\tau_0^{3/2}}  \|g^{(n-2)}\|_{X^{N}_{\tau^{(n-2)}}} \leq \frac{4 \eps N 2^N K}{\tau_0^{3/2} \tt^{5/4-\delta}}   \leq \frac{\delta}{4 K \tt^{5/4}}  \frac{16 \eps N 2^N K^2 \langle T_{\eps,N} \rangle^\delta}{\tau_0^{3/2} \delta} \leq \frac{\delta}{4 K \tt^{5/4}} 
\]
holds in view of the bound \eqref{eq:cc} and the definition of $T_{\eps,N}$ in \eqref{eq:tau:eps:n}. Thus, we have proven \eqref{eq:An:contract}, from which it follows that 
\[
0 \leq A_n \leq a_0 \left(\frac{\sqrt{17}-1}{8}\right)^n  + a_1  \left(\frac{\sqrt{17}+1}{8}\right)^n \to 0 \quad \mbox{as} \quad n \to \infty
\] 
where $a_0,a_1 >0$ are determined from computing $A_1$ and $A_2$. This concludes the proof of convergence for the Picard iteration scheme \eqref{eq:Picard:1}--\eqref{eq:Picard:2} on $[0,T_{\eps,N}]$. The convergence holds in the norm defined by $A_n$. Moreover, the available bounds are sufficient in order to show that the limiting function $\gnu$ obeys \eqref{eq:gnu:1}--\eqref{eq:gnu:4} pointwise in $x$ when integrated against $H^1(\theta_\alpha dy)$ functions of $y$.

\subsection{A posteriori estimates for the dissipative approximation}
Having constructed solutions $\gnu$ of \eqref{eq:gnu:1}--\eqref{eq:gnu:4} with finite Sobolev regularity  in $x$ (of order $N \geq 1/\nu$),  we a posteriori show that these solutions obey better bounds, and in particular, are real-analytic with respect to $x$. 

For this purpose, we would like to perform estimates similar to those in the previous subsection, and pass $N \to \infty$. The main obstruction to directly using the bound \eqref{eq:cc} and passing $N \to \infty$ is that the time of existence we have so far guaranteed for $g^{(\nu)}$ is $T_{\eps,N}$ defined by \eqref{eq:tau:eps:n}, and thus depends on $N$ itself. Thus, the first step is to show that $g^{(\nu)}$ obeys $\nu$-independent Sobolev bounds on a time interval $T_\eps$ that is independent of $N$ (and $\nu$). 

As before, let $N$ be such that
$N \nu \geq 1$ with the caveat that we will in this subsection look for bounds independent of $N$.
Let $K$ be the constant from \eqref{eq:constr:1}, and define $\tau_N^{(\nu)}(t)$ by
\begin{align}
\frac{d}{dt}  {\tau}_N^{(\nu)} 
&+ \frac{4 K}{\tau_0^{1/2}} \left( \| g^{(\nu)}\|_{B^N_{ {\tau}_N^{(\nu)}}}  + \| g^{(\nu)} \|_{\tilde B^N_{\tau_N^{(\nu)}}} \right)  + 16 K^2 \tt^{1/4}  \| g^{(\nu)}\|_{X^N_{ {\tau}_N^{(\nu)}}} \| g^{(\nu)}\|_{ \tilde B^N_{ {\tau}_N^{(\nu)}}} = 0, 
\label{eq:tau:nu:N}
\end{align}
with initial value $\tau_N^{(\nu)}(0) =  7 \tau_0/4$.
For each $N$, this is a  first order ODE, with a degree $N$ polynomial nonlinearity in $\tau^{(\nu)}_N$. Due to the a priori bounds \eqref{eq:cc} inherited by $g^{(\nu)}$, at least on $[0,T_{\eps,N}]$ the ODE \eqref{eq:tau:nu:N} has an absolutely continuous solution. We let $T_{\eps,N}^*$ be the maximal time for which $\tau^{(\nu)}_N$ stays above $5\tau_0/4$. On $[0,T_{\eps,N}^*]$ all the estimates in the previous section are justified.
We already have shown that $T_{\eps,N}^* \geq T_{\eps,N}$, and we now claim that $T_{\eps,N}^* \geq T_\eps$ for the $T_\eps >0$ defined in \eqref{eq:T:eps:*}, which is independent of $N$ and $\nu$.

With $\tau^{(\nu)}_N$ as defined above, and $N \geq \nu^{-1}$ arbitrary, we perform an estimate in the spirit of \eqref{eq:constr:1}--\eqref{eq:constr:2}, use the definition of $\tau_N^{(\nu)}$ in \eqref{eq:tau:nu:N}, and arrive at
\begin{align}
&\frac{d}{dt} \|g^{(\nu)}\|_{X^{N}_{\tau^{(\nu)}_N}} + \frac{5/4-\delta}{\tt} \|g^{(\nu)}\|_{X^{N}_{\tau^{(\nu)}_N}} 
+ \frac{\delta}{K \tt^{5/4}} \left( \|g^{(\nu)}\|_{B^{N}_{\tau^{(\nu)}_N}} + \|g^{(\nu)}\|_{\tilde B^{N}_{\tau^{(\nu)}_N}} \right) \notag\\
&\qquad + \frac{K}{\tau_0^{1/2}} \left( \|g^{(\nu)}\|_{B^{N}_{\tau^{(\nu)}_N}} +  \|g^{(\nu)}\|_{\tilde B^{N}_{\tau^{(\nu)}_N}} \right) \|g^{(\nu)}\|_{Y_{\tau^{(\nu)}_N}^{N}}  
+  \frac{\nu}{8} \|g^{(\nu)}\|_{\tilde Y^{N+1}_{\tau^{(\nu)}_N}} 
\notag\\
& \leq \left( \frac{d}{dt} \tau^{(\nu)}_N  + \frac{K}{\tau_0^{1/2}}  \left( \|g^{(\nu)}\|_{B^{N}_{\tau^{(\nu)}_N}} + \|g^{(\nu)}\|_{\tilde B^{N}_{\tau^{(\nu)}_N}} \right) \right) \|g^{(\nu)}\|_{Y_{\tau^{(\nu)}_N}^{N}} 
- \frac{\nu}{8} \|g^{(\nu)}\|_{\tilde Y^{N+1}_{\tau^{(\nu)}_N}} 
\notag\\
&\qquad 
+ \frac{K}{(\tau^{(\nu)}_N)^{1/2}} \|g^{(\nu)}\|_{B^{N}_{\tau^{(\nu)}_N}} Y_{N+1}(g^{(\nu)},\tau^{(\nu)}_N)
+  \frac{K}{(\tau^{(\nu)}_N)^{1/2}} \|g^{(\nu)}\|_{B^{N}_{\tau^{(\nu)}_N}} \|g^{(\nu)}\|_{Y^{N}_{\tau^{(\nu)}_N}}  
\notag\\
& \leq \left( \frac{d}{dt} \tau^{(\nu)}_N  + \frac{3 K}{\tau_0^{1/2}}  \left( \|g^{(\nu)}\|_{B^{N}_{\tau^{(\nu)}_N}} + \|g^{(\nu)}\|_{\tilde B^{N}_{\tau^{(\nu)}_N}} \right) \right) \|g^{(\nu)}\|_{Y_{\tau^{(\nu)}_N}^{N}} - \frac{\nu}{8} \|g^{(\nu)}\|_{\tilde Y^{N+1}_{\tau^{(\nu)}_N}} 
\notag\\
&\qquad + \frac{\nu}{8} \|g^{(\nu)}\|_{\tilde Y^{N+1}_{\tau^{(\nu)}_N}} 
+ 16 K^2 \tt^{1/4} \|g^{(\nu)}\|_{X^N_{\tau^{(\nu)}_N}} \|g^{(\nu)}\|_{\tilde B^{N}_{\tau^{(\nu)}_N}} \| g^{(\nu)}\|_{Y^N_{\tau^{(\nu)}_N}} 
\notag\\
& \leq 0
\label{eq:constr:5}
\end{align}
where $K \geq 1$ is a universal constant. Integrating the above on $[0,T]$ and using that $\|g_0\|_{X_{2\tau_0}}\leq \eps$, for any $N \geq 1/\nu$ we obtain
\begin{align}
& \sup_{t \in [0,T]} \left( \tt^{5/4-\delta} \|g^{(\nu)}(t)\|_{X^{N}_{\tau^{(\nu)}_N(t)}}\right)
\notag\\
&\quad + \frac{\delta}{K} \int_0^T \frac{1}{\ss^{\delta}} \left( \|g^{(\nu)}(s)\|_{B^{N}_{\tau^{(\nu)}_N(s)}} + \|g^{(\nu)}(s)\|_{\tilde B^{N}_{\tau^{(\nu)}_N(s)}} \right) ds 
\notag\\
&\quad + \frac{K}{\tau_0^{1/2}} \int_0^T  \ss^{5/4-\delta}
\left( \|g^{(\nu)}(s)\|_{B^{N}_{\tau^{(\nu)}_N(s)}} + \|g^{(\nu)}(s)\|_{\tilde B^{N}_{\tau^{(\nu)}_N(s)}} \right) 
\|g^{(\nu)}(s)\|_{Y_{\tau^{(\nu)}_N(s)}^{N}} ds 
\notag\\
&\leq \eps.
\label{eq:constr:6}
\end{align}
Estimate \eqref{eq:constr:6} above implies that
\begin{align}
&\frac{4 K}{\tau_0^{1/2}} \int_0^t \|g^{(\nu)}(s)\|_{B^{N}_{\tau^{(\nu)}_N(s)}} ds \leq \frac{4K^2 \eps \tt^\delta}{\delta \tau_0^{1/2}} = \frac{4 K^2 \tt^\delta}{\tau_0^{1/2} \log \frac 1 \eps}
\label{eq:need:a:beer}\\
& 16K^2 \int_0^t \ss^{1/4} \|g^{(\nu)}(s)\|_{X^{N}_{\tau^{(\nu)}_N(s)}} \|g^{(\nu)}(s)\|_{\tilde B^{N}_{\tau^{(\nu)}_N(s)}} ds \leq \frac{16 K^3 \eps^2}{\delta} = \frac{16 K^3 \eps}{\log \frac{1}{\eps}} \leq \frac{4 K^2 \tt^\delta}{\tau_0^{1/2} \log \frac 1 \eps} \notag
\end{align}
upon appealing to \eqref{eq:eps:tau:0}.
Inserted in \eqref{eq:tau:nu:N}, the above bounds a posteriori show that 
\begin{align}
{\tau}_N^{(\nu)}(t)\geq \frac{7 \tau_0}{4} - \frac{8 K^2 \tt^\delta}{\tau_0^{1/2} \log \frac 1 \eps} \geq \frac{5 \tau_0}{4} \label{eq:tau:nu:N:lower:bnd}
\end{align}
for all $t \leq T_\eps$, as long as $T_\eps$ obeys
\begin{align*}
\langle T_\eps \rangle^\delta \leq \frac{\tau_0^{3/2}  \log \frac 1 \eps}{16 K^2} .
\end{align*}
It is clear that the $T_\eps$ defined earlier in \eqref{eq:T:eps:*} obeys the above estimate if $K_*$ is taken sufficiently large.
This shows that $T_{\eps,N}^* \geq T_\eps$ for each $N \geq \nu^{-1}$.
Moreover, the  bound \eqref{eq:tau:nu:N:lower:bnd} which combined with \eqref{eq:constr:6} yields
\begin{align}
& \sup_{t \in [0,T_\eps]} \left( \tt^{5/4-\delta} \|g^{(\nu)}(t)\|_{X^{N}_{5 \tau_0/4}}\right) + \frac{\delta}{K} \int_0^{T_\eps} \frac{1}{\ss^{\delta}} \left( \|g^{(\nu)}(s)\|_{B^{N}_{5 \tau_0/4}} + \|g^{(\nu)}(s)\|_{\tilde B^{N}_{5 \tau_0/4}} \right) ds 
\notag\\
&\quad + \frac{K}{\tau_0^{1/2}} \int_0^{T_\eps} \ss^{5/4-\delta}
\left( \|g^{(\nu)}(s)\|_{B^{N}_{5 \tau_0/4}} + \|g^{(\nu)}(s)\|_{\tilde B^{N}_{5 \tau_0/4}} \right) 
\|g^{(\nu)}(s)\|_{Y_{5 \tau_0/4}^{N}} ds 
\notag\\
&\leq \eps,
\label{eq:constr:7}
\end{align}
for any $N\geq 1$, 
 where $K\geq 1$ is a fixed universal constant. Note that upon passing $N\to \infty$ in \eqref{eq:constr:7}, and using the Monotone Convergence Theorem, we also obtain the bound
 \begin{align}
& \sup_{t \in [0,T_\eps]} \left( \tt^{5/4-\delta} \|g^{(\nu)}(t)\|_{X_{5 \tau_0/4}}\right) + \frac{\delta}{K} \int_0^{T_\eps} \frac{1}{\ss^{\delta}} \left( \|g^{(\nu)}(s)\|_{B_{5 \tau_0/4}} + \|g^{(\nu)}(s)\|_{\tilde B_{5 \tau_0/4}} \right) ds 
\notag\\
&\quad + \frac{K}{\tau_0^{1/2}} \int_0^{T_\eps} \ss^{5/4-\delta}  
\left( \|g^{(\nu)}(s)\|_{B_{5 \tau_0/4}} + \|g^{(\nu)}(s)\|_{\tilde B_{5 \tau_0/4}} \right) 
\|g^{(\nu)}(s)\|_{Y_{5 \tau_0/4}} ds 
\notag\\
&\leq 4 \eps,
\label{eq:constr:7:infty}
\end{align}
for the real-analytic norms of $g^{(\nu)}$. Due to the monotonicity of the norms with respect to $\tau$, this proves \eqref{eq:gnu:bnd}.

In order to obtain a limiting analyticity radius $\tau^{(\nu)}$ in the limit as $N \to \infty$, which obeys the nonlinear ODE \eqref{eq:tau:nu}, we may first try to show that the sequence of absolutely continuous functions $\{ \tau^{(\nu)}_N \}_{N \geq \nu^{-1}}$ is in fact equicontinuous on the time interval $[0,T_\eps]$. This seems however not possible due to the third term on the left side of \eqref{eq:tau:nu:N}. We instead define a new sequence of radii $\theta^{(\nu)}_N$, for which the trick used to prove uniqueness in Section~\ref{sec:uniqueness} applies, and we are able to prove that $\{\theta^{(\nu)}_N\}_{N\geq \nu^{-1}}$ is uniformly equicontinuous (in fact uniformly H\"older-$1/2$ in time). Let
\begin{align}
\frac{d}{dt} \theta^{(\nu)}_N + \frac{2 K}{(\theta^{(\nu)}_N)^{1/2}} \| g^{(\nu)} \|_{B^{N}_{\theta^{(\nu)}_N}} =  0, \qquad \theta^{(\nu)}_N(0) = \tau_0.
\label{eq:theta:n:def}
\end{align}
The existence of solutions to \eqref{eq:theta:n:def} is immediate since the nonlinearity is a polynomial of finite degree, with coefficients that are integrable in time by \eqref{eq:constr:7:infty}.
We next observe that in view of \eqref{eq:need:a:beer}, by using a version of \eqref{eq:tau:nu:N:lower:bnd}, we arrive at
\begin{align*}
\theta^{(\nu)}_N(t) \geq \frac{\tau_0}{2} \quad \mbox{for all} \quad t \in [0,T_\eps].
\end{align*}
Now, similarly to \eqref{eq:constr:5} we have that
\begin{align}
&\frac{d}{dt} \|g^{(\nu)}\|_{X^{N}_{\theta^{(\nu)}_N}} + \frac{5/4-\delta}{\tt} \|g^{(\nu)}\|_{X^{N}_{\theta^{(\nu)}_N}} 
+ \frac{\delta}{K \tt^{5/4}}  \|g^{(\nu)}\|_{B^{N}_{\theta^{(\nu)}_N}} + \frac{\delta}{2 K \tt^{5/4}} \|g^{(\nu)}\|_{\tilde B^{N}_{\theta^{(\nu)}_N}}
  \notag\\
&\qquad + \frac{ K}{(\theta^{(\nu)}_N)^{1/2}}  \|g^{(\nu)}\|_{B^{N}_{\theta^{(\nu)}_N}}  \|g^{(\nu)}\|_{Y_{\theta^{(\nu)}_N}^{N}}  
+  \frac{\nu}{8} \|g^{(\nu)}\|_{\tilde Y^{N+1}_{\theta^{(\nu)}_N}} \notag\\
& \leq \left( \frac{d}{dt} \theta^{(\nu)}_N  + \frac{ K}{(\theta^{(\nu)}_N)^{1/2}}    \|g^{(\nu)}\|_{B^{N}_{\theta^{(\nu)}_N}}  \right) \|g^{(\nu)}\|_{Y_{\theta^{(\nu)}_N}^{N}} 
- \frac{\nu}{8} \|g^{(\nu)}\|_{\tilde Y^{N+1}_{\theta^{(\nu)}_N}}  - \frac{\delta}{2 K \tt^{5/4}} \|g^{(\nu)}\|_{\tilde B^{N}_{\theta^{(\nu)}_N}}
\notag\\
&\qquad 
+ \frac{K}{(\theta^{(\nu)}_N)^{1/2}} \|g^{(\nu)}\|_{B^{N}_{\theta^{(\nu)}_N}} Y_{N+1}(g^{(\nu)},\theta^{(\nu)}_N)
+  \frac{K}{(\theta^{(\nu)}_N)^{1/2}} \|g^{(\nu)}\|_{B^{N}_{\theta^{(\nu)}_N}} \|g^{(\nu)}\|_{Y^{N}_{\theta^{(\nu)}_N}}  
\notag\\
& \leq \left( \frac{d}{dt} \theta^{(\nu)}_N  + \frac{2 K}{(\theta^{(\nu)}_N)^{1/2}}  \left( \|g^{(\nu)}\|_{B^{N}_{\theta^{(\nu)}_N}} + \|g^{(\nu)}\|_{\tilde B^{N}_{\theta^{(\nu)}_N}} \right) \right) \|g^{(\nu)}\|_{Y_{\theta^{(\nu)}_N}^{N}} - \frac{\delta}{2 K \tt^{5/4}} \|g^{(\nu)}\|_{\tilde B^{N}_{\theta^{(\nu)}_N}}
\notag\\
&\qquad + \frac{16 K^2 \tt^{1/4}}{\nu N}  \|g^{(\nu)}\|_{X^N_{\tau_0}} \|g^{(\nu)}\|_{\tilde B^{N}_{\theta^{(\nu)}_N}} \| g^{(\nu)}\|_{Y^N_{\tau_0}}
\notag\\
& \leq \left( \frac{16 K^3 \tt^{1/4}}{\nu N \tau_0}  \|g^{(\nu)}\|_{X^N_{5\tau_0/4}}^2 - \frac{\delta}{2 K\tt^{5/4}} \right) \|g^{(\nu)}\|_{\tilde B^{N}_{\theta^{(\nu)}_N}}
\notag\\
&\leq 0.
\label{eq:constr:8}
\end{align}
In the last inequality of \eqref{eq:constr:8} we have used the same trick as in Section~\ref{sec:uniqueness}: that by \eqref{eq:constr:7} we have
\begin{align*}
\sup_{t\in [0,T_\eps]} \left( \frac{32 K^4}{\tau_0} \tt^{3/2} \|g^{(\nu)}\|_{X_{5 \tau_0/4}}^2 \right) \leq  \frac{32 K^4 \eps^2}{\tau_0} \leq \eps \log \frac{1}{\eps} = \delta
\end{align*}
upon appealing to assumption \eqref{eq:eps:tau:0}. 

Using the bound \eqref{eq:constr:8}, we show that the sequence of absolutely continuous functions $\{\theta^{(\nu)}_N\}_{N\geq \nu^{-1}}$, is in fact uniformly bounded in $C^{1/2}([0,T_\eps])$, and thus uniformly equicontinuous. For this purpose, let $t_1,t_2 \in [0,T_\eps]$ be such that $|t_1-t_2|\leq \zeta$. Using the mean value theorem, the definition of $\theta^{(\nu)}_N$ in \eqref{eq:theta:n:def}, and the definitions \eqref{eq:XmN}--\eqref{eq:tildeBmN} we arrive at 
\begin{align}
| \theta^{(\nu)}_N(t_1) - \theta^{(\nu)}_N(t_2) | 
&\leq \frac{4 K}{\tau_0^{1/2}} \int_{t_1}^{t_2} \| g^{(\nu)}(s)\|_{B^{N}_{\theta_N^{(\nu)}(s)}} ds
\notag\\
&\leq \frac{4 K}{\tau_0^{1/2}} \int_{t_1}^{t_2} \ss^{1/8} \| g^{(\nu)}(s)\|_{X^{N}_{\theta_N^{(\nu)}(s)}}^{1/2} \| g^{(\nu)}(s)\|_{\tilde B^{N}_{\theta_N^{(\nu)}(s)}}^{1/2} ds
\notag\\
&\leq \frac{16 K^2|t_1-t_2|}{\zeta^{1/2} \tau_0} + \zeta^{1/2} \int_{t_1}^{t_2} \ss^{1/4} \| g^{(\nu)}(s)\|_{X^{N}_{\theta_N^{(\nu)}(s)}} \| g^{(\nu)}(s)\|_{\tilde B^{N}_{\theta_N^{(\nu)}(s)}} ds 
\notag\\
&\leq \left( \frac{16 K^2 }{\tau_0} + \frac{16 K \eps^2}{\delta} \right) \zeta^{1/2}
\label{eq:constr:9}
\end{align}
Since $\zeta \in (0,1)$ was arbitrary,  it follows from \eqref{eq:constr:9} that the $\theta^{(\nu)}_N$ are uniformly equicontinous. The Arzela-Ascoli theorem guarantees the existence of a subsequence $\theta^{(\nu)}_{N_k}$ with $N_k \to \infty$ as $k \to \infty$, and of a function $\tau^{(\nu)}$ such that
\begin{align*}
\theta^{(\nu)}_{N_k} \to \tau^{(\nu)} \quad \mbox{uniformly on} \quad [0,T_\eps] \quad \mbox{as} \quad k \to \infty.
\end{align*}
Moreover, we have that $\tau^{(\nu)} \geq \tau_0/2$ on $[0,T_\eps]$. By passing $N=N_k \to \infty$ in \eqref{eq:theta:n:def} we obtain 
\begin{align}
\frac{d}{dt} \tau^{(\nu)} + \frac{2K}{(\tau^{(\nu)})^{1/2}} \| g^{(\nu)} \|_{B_{\tau^{(\nu)}}} = 0, \qquad \tau^{(\nu)}(0) = \tau_0.
\label{eq:tau:nu:def:2}
\end{align}
In order to justify \eqref{eq:tau:nu:def:2} we use that by~\eqref{eq:constr:7:infty} we have that $\|g^{(\nu)}\|_{B_{5 \tau_0/4}} \in L^1([0,T_\eps])$, and that the convergence of $\tau^{(\nu)}_{N_k} \to \tau^{(\nu)}$ is uniform. 
Moreover, using \eqref{eq:constr:7:infty} and a bound similar to \eqref{eq:constr:9}, it follows from \eqref{eq:tau:nu:def:2} that 
\begin{align*}
| \tau^{(\nu)}(t_1) - \tau^{(\nu)}(t_2) | \leq \left( \frac{16 K^2 }{\tau_0} + \frac{16 K \eps^2}{\delta} \right) |t_1-t_2|^{1/2}
\end{align*}
uniformly for $t_1,t_2 \in [0,T_\eps]$.
That is, the radii $\tau^{(\nu)}$ are uniformly (with respect to $\nu$) H\"older $1/2$ continuous. This concludes the proof of \eqref{eq:gnu:bnd}--\eqref{eq:tau:nu}.

\subsection{Existence of solutions to the Prandtl system}

It remains to pass $\nu \to 0$ and obtain a limiting solution $g$ of the Prandtl equations \eqref{eq:g:1}--\eqref{eq:g:3}, in the sense of Definition~\ref{def:sol}, of a tangential analyticity radius $\tau$ which solves \eqref{eq:tau:1}, such that the pair $(g,\tau)$ obeys the bounds \eqref{eq:g:bnd:0}--\eqref{eq:g:bnd:3}.

We have shown in the previous subsection that the sequence $\tau^{(\nu)}$ is uniformly equicontinuous, and thus by the Arzela-Ascoli theorem we know that along a subsequence $\nu_k \to 0$, we have that $\tau^{(\nu_k)} \to \tau$ uniformly on $[0,T_\eps]$, with $\tau(0)=\tau_0$, and $\tau(t) \geq \tau_0/2$ on this interval. Note that from the bound \eqref{eq:constr:7:infty} it follows that \eqref{eq:gnu:bnd} holds with $\tau^{(\nu)}$ replaced by the limiting function $\tau$.

Without loss of generality, the above subsequence $\{ \nu_k\}_{k\geq 1}$ obeys 
\[
0 < \frac{1}{\nu_{k+1}} - \frac{1}{\nu_k} \leq \frac{1}{k^2}.
\] 
We next show that the subsequence $g^{(\nu_k)}$ is Cauchy in the norm induced by the left side of \eqref{eq:gnu:bnd}. For this purpose, let 
\[ \bar g_k = g^{(\nu_k)} - g^{(\nu_{k+1})}\]
and define
\begin{align*}
G_k 
:=&\sup_{[0,T_{\eps}]} \left( \tt^{5/4-\delta} \| \bar g_k(t)\|_{X_{\bar \tau_k(t)}} \right) + \frac{\delta}{2K} \int_0^{T_{\eps}} \frac{1}{\ss^\delta}  \|\bar g_k(s)\|_{B_{\bar \tau_k(s)}}   ds  \notag \\
&\  + K \int_0^{T_{\eps}} \frac{\ss^{5/4-\delta}}{\bar \tau_k (s)^{1/2}} 
\|g^{(\nu_{k})}(s)\|_{B_{5\tau_0/4}} \|\bar g_k\|_{Y_{\bar \tau_k(s)}} ds
\end{align*}
where 
\begin{align} 
\frac{d}{dt} \bar \tau_k + \frac{2K}{\bar \tau_k^{1/2}}   \| g^{(\nu_k)} \|_{B_{5\tau_0/4}} = 0, \qquad \bar \tau_k (0) = \frac{6 \tau_0}{5}.
\label{eq:one:more:radius}
\end{align}
Note that from the bound \eqref{eq:constr:7:infty}, upon choosing $K_*$ sufficiently large, we obtain that
\begin{align*} 
\bar \tau_k(t) \geq \tau_0 \quad \mbox{for all} \quad t \in [0,T_\eps].
\end{align*}
We claim that 
\begin{align} 
G_k \leq \frac{1}{k^2 \tau_0^2}.
\label{eq:to:do:last}
\end{align}
To prove \eqref{eq:to:do:last}, we consider the equation obeyed by $\bar g_k$
\begin{align*} 
&\partial_t \bar g_k - \partial_y^2 \bar g_k + \kappa \phi \partial_x \bar g_k + \frac{1}{\tt} \bar g_k - \nu_k \partial_x^2 \bar g_k\notag\\
&\quad = (\nu_k - \nu_{k+1}) \partial_x^2 g^{(\nu_{k+1})} 
- \uu(g^{(\nu_k)}) \partial_x \bar g_k - \uu(\bar g_k) \partial_x g^{(\nu_{k+1})} \notag\\
& \qquad - \vv(g^{(\nu_{k+1})}) \partial_y \bar g_k - \vv(\bar g_k) \partial_y g^{(\nu_{k})}
+ \frac{1}{2\tt} \uu(g^{(\nu_k)}) \vv(\bar g_k) +\frac{1}{2\tt} \uu(\bar g_k) \vv(g^{(\nu_{k+1})}).
\end{align*}
Similarly to \eqref{eq:diff:2} we obtain that 
\begin{align}
&\frac{d}{dt} \| \bar g_k\|_{X_{\bar \tau_k(t)}} + \frac{5/4-\delta}{\tt} \| \bar g_k\|_{X_{\bar \tau_k(t)}} + \frac{\delta}{K \tt^{5/4}} \| \bar g_k\|_{B_{\bar \tau_k(t)}} 
 + \frac{\nu_k}{4} \| \bar g_k\|_{\tilde Y_{\bar \tau_k(t)}} + \frac{K}{\bar \tau_k^{1/2}} \| g^{(\nu_k)} \|_{B_{\bar \tau_k}} \| \bar g_k\|_{Y_{\bar \tau_k(t)}} 
\notag\\
&\qquad \leq \left( \frac{d}{dt} \bar \tau_k + \frac{2K}{\bar \tau_k^{1/2}} \| g^{(\nu_k)} \|_{B_{\bar \tau_k}} \right) \| \bar g_k\|_{Y_{\bar \tau_k(t)}} 
 + |\nu_k - \nu_{k+1}| \|\partial_x^2 g^{(\nu_{k+1})}\|_{X_{\bar \tau_k}} + \frac{K}{\bar \tau_k^{1/2}} \|\bar g_k\|_{B_{\bar \tau_k}}
\|g^{(\nu_{k+1})}\|_{Y_{\bar \tau_k}} 
\notag\\
&\qquad \leq \left( \frac{d}{dt} \bar \tau_k + \frac{2K}{\bar \tau_k^{1/2}} \| g^{(\nu_k)} \|_{B_{5\tau_0/4}} \right) \| \bar g_k\|_{Y_{\bar \tau_k(t)}} 
+ \frac{K}{k^2 \bar \tau_k^{2}} \|g^{(\nu_{k+1})}\|_{X_{5\tau_0/4}} 
+ \frac{K^2}{\bar \tau_k^{3/2}} \|\bar g_k\|_{B_{\bar \tau_k}}  \|g^{(\nu_{k+1})}\|_{X_{5\tau_0/4}} 
\label{eq:constr:11}.
\end{align}
In the last inequality above we used that \eqref{eq:one:more:radius} implies that $\bar \tau_k \leq 6 \tau_0/5 < 5 \tau_0 / 4$, and thus
\[
\left( \sup_{m\geq 0} \left( \frac{\bar \tau_k}{5 \tau_0/4 } \right)^{m+2} \frac{M_m}{M_{m+2}} \right) + \left( \sup_{m\geq 0} \left( \frac{\bar \tau_k}{5 \tau_0/4 } \right)^{m} m \right) \leq K
\]
upon possibly increasing the value of $K$. Using the definition of $\bar \tau_k$ we obtain from \eqref{eq:constr:11} that 
\begin{align}
&\frac{d}{dt} \| \bar g_k\|_{X_{\bar \tau_k(t)}} + \frac{5/4-\delta}{\tt} \| \bar g_k\|_{X_{\bar \tau_k(t)}} + \frac{\delta}{2 K \tt^{5/4}} \| \bar g_k\|_{B_{\bar \tau_k(t)}} + \frac{K}{\bar \tau_k^{1/2}} \| g^{(\nu_k)} \|_{B_{\bar \tau_k}} \| \bar g_k\|_{Y_{\bar \tau_k(t)}} 
\notag\\
&\qquad +\frac{\delta}{2 K \tt^{5/4}}
\left(1 - \frac{2 K^3}{\delta \tau_0^{3/2}} \tt^{5/4} \|g^{(\nu_{k+1})}\|_{X_{5\tau_0/4}} \right) \| \bar g_k\|_{B_{\bar \tau_k(t)}} 
\notag\\
&\  \leq \frac{K}{k^2 \tau_0^2}  \|g^{(\nu_{k+1})}\|_{X_{5\tau_0/4}}.
\label{eq:constr:12}
\end{align}
At this stage, we use \eqref{eq:constr:7:infty} and \eqref{eq:T:eps:*} which imply that 
\begin{align*}
\sup_{t\in [0,T_\eps]} \left( \frac{2 K^3}{\delta \tau_0^{3/2}} \tt^{5/4} \|g^{(\nu_{k+1})}(t)\|_{X_{5\tau_0/4}}  \right)\leq \frac{4 K^3}{\delta \tau_0^{3/2}} \eps \langle T_\eps \rangle^\delta = \frac{4 K^3}{\tau_0^{3/2} \log \frac{1}{\eps}} \langle T_\eps \rangle^\delta
\leq 1
\end{align*}
if $K_*$ is sufficiently large.
Then, again appealing to \eqref{eq:constr:7:infty}, upon integrating \eqref{eq:constr:12} in time we obtain
\begin{align*} 
G_k \leq \frac{\eps K}{k^2 \tau_0^2} \leq \frac{1}{k^2 \tau_0^2}.
\end{align*}
From here it follows that $g^{(\nu_k)}$ is a Cauchy sequence in the topology induced by $G_k$. 

Thus, there exists a limiting function $g \in L^\infty([0,T_\eps];X_{\tau_0}) \cap L^1([0,T_\eps];B_{\tau_0}\cap \tilde B_{\tau_0})$ such that $g^{(\nu_k)} \to g$ in this norm. In particular, it immediately follows that  $g \in L^\infty([0,T_\eps];\HCal_{2,1,\alpha/\tt})$ and thus $g$ is a solution of the Prandtl equations \eqref{eq:g:1}--\eqref{eq:g:3} in the sense of Definition~\ref{def:sol}. 

Finally, using these bounds one may show that upon passing $\nu = \nu_k \to 0$ in \eqref{eq:tau:nu:def:2}, the limiting analyticity radius $\tau(t)$ and the solution $g(t)$ obey the ODE \eqref{eq:tau:1}. This concludes the proof of the existence of solutions to \eqref{eq:g:1}--\eqref{eq:g:3}.


\subsection*{Acknowledgments} 
The authors are thankful to Igor Kukavica, Toan Nguyen, and Marius Paicu for stimulating discussions, and to the anonymous referees for helpful suggestions. The work of VV was supported in part by the NSF grants DMS-1348193, DMS-1514771, and an Alfred P. Sloan Fellowship. 

\subsection*{Conflict of Interest} The authors declare that they have no conflict of interest.



\begin{thebibliography}{KMVW14}

\bibitem[AWXY14]{AlexandreWangXuYang14}
R.~Alexandre, Y.-G. Wang, C.-J. Xu, and T.~Yang.
\newblock Well-posedness of the {P}randtl equation in {S}obolev spaces.
\newblock {\em J. Amer. Math. Soc.}, 2014.

\bibitem[CGP11]{CheminGallagherPaicu11}
J.-Y. Chemin, I.~Gallagher, and M.~Paicu.
\newblock Global regularity for some classes of large solutions to the
  {N}avier-{S}tokes equations.
\newblock {\em Ann. of Math. (2)}, 173(2):983--1012, 2011.

\bibitem[CKV14]{ConstantinKukavicaVicol14}
P.~Constantin, I.~Kukavica, and V.~Vicol.
\newblock On the inviscid limit of the {N}avier-{S}tokes equations.
\newblock {\em arXiv:1403.5748, Proc. Amer. Math. Soc.}, 2014.

\bibitem[CLS01]{CannoneLombardoSammartino01}
M.~Cannone, M.C. Lombardo, and M.~Sammartino.
\newblock Existence and uniqueness for the {P}randtl equations.
\newblock {\em C. R. Acad. Sci. Paris S{\'e}r. I Math.}, 332(3):277--282, 2001.

\bibitem[CS00]{CaflischSammartino00}
R.E. Caflisch and M.~Sammartino.
\newblock Existence and singularities for the {P}randtl boundary layer
  equations.
\newblock {\em ZAMM Z. Angew. Math. Mech.}, 80(11-12):733--744, 2000.

\bibitem[DR04]{DrazinReid04}
P.G. Drazin and W.H. Reid.
\newblock {\em Hydrodynamic stability}.
\newblock Cambridge University Press, 2004.

\bibitem[EE97]{EEngquist97}
W.~E and B.~Engquist.
\newblock Blowup of solutions of the unsteady {P}randtl's equation.
\newblock {\em Comm. Pure Appl. Math.}, 50(12):1287--1293, 1997.

\bibitem[GGN14a]{GrenierGuoNguyen14c}
E.~Grenier, Y.~Guo, and T.~Nguyen.
\newblock Spectral instability of characteristic boundary layer flows.
\newblock {\em arXiv:1406.3862}, 2014.

\bibitem[GGN14b]{GrenierGuoNguyen14}
E.~Grenier, Y.~Guo, and T.~Nguyen.
\newblock Spectral instability of symmetric shear flows in a two-dimensional
  channel.
\newblock {\em arXiv:1402.1395}, 2014.

\bibitem[GGN14c]{GrenierGuoNguyen14b}
E.~Grenier, Y.~Guo, and T.~Nguyen.
\newblock Spectral stability of {P}randtl boundary layers: an overview.
\newblock {\em arXiv:1406.4452}, 2014.

\bibitem[GN11]{GuoNguyen11}
Y.~Guo and T.~Nguyen.
\newblock A note on {P}randtl boundary layers.
\newblock {\em Comm. Pure Appl. Math.}, 64(10):1416--1438, 2011.

\bibitem[GN14]{GuoNguyen14}
Y.~Guo and T.~Nguyen.
\newblock {Prandtl boundary layer expansions of steady Navier-Stokes flows over
  a moving plate}.
\newblock {\em arXiv:1411.6984}, 2014.

\bibitem[Gre00a]{Grenier00}
E.~Grenier.
\newblock On the nonlinear instability of {E}uler and {P}randtl equations.
\newblock {\em Comm. Pure Appl. Math.}, 53(9):1067--1091, 2000.

\bibitem[Gre00b]{Grenier00b}
E.~Grenier.
\newblock On the stability of boundary layers of incompressible {E}uler
  equations.
\newblock {\em J. Differential Equations}, 164(1):180--222, 2000.

\bibitem[GSS09]{GarganoSammartinoSciacca09}
F.~Gargano, M.~Sammartino, and V.~Sciacca.
\newblock Singularity formation for {P}randtl's equations.
\newblock {\em Phys. D}, 238(19):1975--1991, 2009.

\bibitem[GVD10]{GerardVaretDormy10}
D.~G{{\'e}}rard-Varet and E.~Dormy.
\newblock On the ill-posedness of the {P}randtl equation.
\newblock {\em J. Amer. Math. Soc.}, 23(2):591--609, 2010.

\bibitem[GVM13]{GerardVaretMasmoudi13}
D.~G{{\'e}}rard-Varet and N.~Masmoudi.
\newblock Well-posedness for the {P}randtl system without analyticity or
  monotonicity.
\newblock {\em arXiv:1305.0221}, 2013.

\bibitem[GVN12]{GerardVaretNguyen12}
D.~G\'{e}rard-Varet and T.~Nguyen.
\newblock Remarks on the ill-posedness of the {P}randtl equation.
\newblock {\em Asymptotic Analysis}, 77:71--88, 2012.

\bibitem[HH03]{HongHunter03}
L.~Hong and J.K. Hunter.
\newblock Singularity formation and instability in the unsteady inviscid and
  viscous {P}randtl equations.
\newblock {\em Commun. Math. Sci.}, 1(2):293--316, 2003.

\bibitem[H{\"o}r83]{Hormander83}
L.~H{\"o}rmander.
\newblock {\em The analysis of linear partial differential operators III},
  volume 257.
\newblock Springer, 1983.

\bibitem[IKZ12]{IgnatovaKukavicaZiane12}
M.~Ignatova, I.~Kukavica, and M.~Ziane.
\newblock Local existence of solutions to the free boundary value problem for
  the primitive equations of the ocean.
\newblock {\em Journal of Mathematical Physics}, 53:103101, 2012.

\bibitem[Kla83]{Klainerman83}
S.~Klainerman.
\newblock On ``almost global'' solutions to quasilinear wave equations in three
  space dimensions.
\newblock {\em Comm. Pure Appl. Math.}, 36(3):325--344, 1983.

\bibitem[KMVW14]{KukavicaMasmoudiVicolWong14}
I.~Kukavica, N.~Masmoudi, V.~Vicol, and T.K. Wong.
\newblock On the local well-posedness of the {P}randtl and the hydrostatic
  {E}uler equations with multiple monotonicity regions.
\newblock {\em SIAM J. Math. Anal.}, 46(6):3865--3890, 2014.

\bibitem[KTVZ11]{KukavicaTemamVicolZiane11}
I.~Kukavica, R.~Temam, V.~Vicol, and M.~Ziane.
\newblock Local existence and uniqueness for the hydrostatic {E}uler equations
  on a bounded domain.
\newblock {\em J. Differential Equations}, 250(3):1719--1746, 2011.

\bibitem[KV11]{KukavicaVicol11a}
I.~Kukavica and V.C. Vicol.
\newblock The domain of analyticity of solutions to the three-dimensional
  {E}uler equations in a half space.
\newblock {\em Discrete Contin. Dyn. Syst.}, 29(1):285--303, 2011.

\bibitem[KV13]{KukavicaVicol13a}
I.~Kukavica and V.~Vicol.
\newblock On the local existence of analytic solutions to the {P}randtl
  boundary layer equations.
\newblock {\em Commun. Math. Sci.}, 11(1):269--292, 2013.

\bibitem[LCS03]{LombardoCannoneSammartino03}
M.C. Lombardo, M.~Cannone, and M.~Sammartino.
\newblock Well-posedness of the boundary layer equations.
\newblock {\em SIAM J. Math. Anal.}, 35(4):987--1004 (electronic), 2003.

\bibitem[LWX15]{LiWuXu15}
W.~Li, D.~Wu, and C.-J. Xu.
\newblock Gevrey class smoothing effect for the {P}randtl equation.
\newblock {\em arXiv:1502.03569}, 02 2015.

\bibitem[LWY14]{LiuWangYang14}
C.-J. Liu, Y.-G. Wang, and T.~Yang.
\newblock A well-posedness theory for the {P}randtl equations in three space
  variables.
\newblock {\em arXiv:1405.5308}, 05 2014.

\bibitem[Mae14]{Maekawa14}
Y.~Maekawa.
\newblock On the inviscid limit problem of the vorticity equations for viscous
  incompressible flows in the half-plane.
\newblock {\em Comm. Pure Appl. Math.}, 67(7):1045--1128, 2014.

\bibitem[MW12]{MasmoudiWong12b}
N.~Masmoudi and T.K. Wong.
\newblock On the ${H}^s$ theory of hydrostatic {E}uler equations.
\newblock {\em Arch. Ration. Mech. Anal.}, 204(1):231--271, 2012.

\bibitem[MW14]{MasmoudiWong12a}
N.~Masmoudi and T.K. Wong.
\newblock Local-in-time existence and uniqueness of solutions to the {P}randtl
  equations by energy methods.
\newblock {\em arXiv:1206.3629, Comm. Pure Appl.~Math.,}, 2014.

\bibitem[Ole66]{Oleinik66}
O.A. Ole{\u\i}nik.
\newblock On the mathematical theory of boundary layer for an unsteady flow of
  incompressible fluid.
\newblock {\em J. Appl. Math. Mech.}, 30:951--974 (1967), 1966.

\bibitem[OT01]{OliverTiti01}
M.~Oliver and E.S. Titi.
\newblock On the domain of analyticity of solutions of second order analytic
  nonlinear differential equations.
\newblock {\em J. Differential Equations}, 174(1):55--74, 2001.

\bibitem[Pra04]{Prandtl1904}
L.~Prandtl.
\newblock {\"{U}}ber {F}l{\"u}ssigkeitsbewegung bei sehr kleiner {R}eibung.
\newblock {\em Verh. III Intern. Math. Kongr. Heidelberg, Teuber, Leipzig},
  pages 485--491, 1904.

\bibitem[PV11]{PaicuVicol11}
M.~Paicu and V.~Vicol.
\newblock Analyticity and gevrey-class regularity for the second-grade fluid
  equations.
\newblock {\em J. Math. Fluid Mech.}, 13(4):533--555, 2011.

\bibitem[PZ14]{PaicuZhang14}
M.~Paicu and Z.~Zhang.
\newblock Global well-posedness for 3{D} {N}avier-{S}tokes equations with
  ill-prepared initial data.
\newblock {\em J. Inst. Math. Jussieu}, 13(2):395--411, 2014.

\bibitem[SC98]{SammartinoCaflisch98a}
M.~Sammartino and R.E. Caflisch.
\newblock Zero viscosity limit for analytic solutions, of the {N}avier-{S}tokes
  equation on a half-space. {I}. {E}xistence for {E}uler and {P}randtl
  equations.
\newblock {\em Comm. Math. Phys.}, 192(2):433--461, 1998.

\bibitem[WXY14]{WangXieYang14}
Y.-G. Wang, F.~Xie, and T.~Yang.
\newblock Local well-posedness of {P}randtl equations for compressible flow in
  two space variables.
\newblock {\em arXiv:1407.3637}, 07 2014.

\bibitem[XZ04]{XinZhang04}
Z.~Xin and L.~Zhang.
\newblock On the global existence of solutions to the {P}randtl's system.
\newblock {\em Adv. Math.}, 181(1):88--133, 2004.

\bibitem[ZZ14]{ZhangZhang14}
P.~Zhang and Z.~Zhang.
\newblock Long time well-posdness of {P}randtl system with small and analytic
  initial data.
\newblock {\em arXiv:1409.1648}, 2014.

\end{thebibliography}
\end{document}